\documentclass[american,10pt]{amsart}
\usepackage[utf8]{inputenc}
\usepackage{amssymb}
\usepackage{graphicx}
\graphicspath{ {Figures/} }
\usepackage{mathtools}
\usepackage{enumitem}
\usepackage[pdfencoding=auto, psdextra]{hyperref}
\usepackage{todonotes}

\usepackage{appendix}
\usepackage[foot]{amsaddr}
\usepackage{cite}

\newtheorem{theorem}{Theorem}
\newtheorem{proposition}[theorem]{Proposition}
\newtheorem{lemma}[theorem]{Lemma}
\newtheorem{corollary}[theorem]{Corollary}
\newtheorem{definition}[theorem]{Definition}

\title[Agro-ecological control: preventing spreading]{Agro-ecological control of a pest--host system: preventing spreading}

\author{Léo Girardin}
\address[L. G.]{CNRS, Institut Camille Jordan, Universit\'{e} Claude Bernard Lyon-1, 43 boulevard du 11 novembre 1918, 69622 Villeurbanne Cedex, France}
\email{leo.girardin@math.cnrs.fr}

\author{Baptiste Maucourt}
\address[B. M.]{Institut Camille Jordan, Universit\'{e} Claude Bernard Lyon-1, 43 boulevard du 11 novembre 1918, 69622 Villeurbanne Cedex, France}
\email{maucourt@math.univ-lyon1.fr}

\thanks{L. G. acknowledges support from the ANR via the project Indyana under grant agreement ANR-21-CE40-0008.}

\begin{document}


\maketitle

\begin{abstract}
We consider an agro-ecologically motivated coupling between a prey--predator system and a vector-borne epidemic system. The coupled system contains one ODE, two reaction--diffusion PDEs and one reaction--diffusion--advection PDE; it has no complete variational or monotonic structure and coefficients are spatially heterogeneous. We study the long-time behavior of solutions of the Cauchy--Neumann problem, which turns out to be largely decided by a linear stability criterion but still involves nonlinear subtleties. Then we consider an optimal control problem. Although it remains elusive analytically, we show numerically the variety of outcomes and the strong impact of the initial conditions.
\end{abstract}

\section{Introduction}

We are interested in the sugar beet agro-ecosystem. In the beet field (which is typically a rectangle), aphids (namely, Myzus Persicae and Aphis Fabae) are vectors of four yellows viruses: the beet mild yellowing virus (BMYV), the beet chlorosis virus (BChV), the beet yellows virus (BYV) and the beet mosaic virus (BtMV). The spread of these viruses in a field can be an economic disaster if left uncontrolled, as the sugar concentration in an infected beet decreases by about $50\%$.

To date, pesticides are the main tool to protect beets, and are used by a large majority of farmers. However, these pesticides have been proven to be a threat to biodiversity and human health. Finding effective alternatives to their use is a major challenge for a modern and sustainable world.

In our system, natural predators of the aphids (for instance, ladybugs Hippodamia variegata or Chnootriba similis) act as a biological method to control the aphid population. Usually, the lack of plants or flowers creates an unsuitable environment for such predators. We sacrifice a portion of the sugar beets in favor of flowers that will form what we call biodiversity refuges. In these flower patches, we expect an abundance of insects that ladybugs can feed on and use to reproduce faster. These refuges should help the predators to thrive and control the aphids, but might also help the aphids to grow faster, as the sugar beets are initially not developped enough for the aphids to feed on them. If this turns out to be biologically incorrect, the analysis we have done works just the same (the growth rate in the refuges would just be smaller than in the field). The refuges will in any case reduce the quantity of beets in the field. Therefore we expect conditions on the ecological properties of these refuges as well as the existence of an optimal refuge geometry. We understand this problem
as an optimal control problem.

Many other agro-ecosystems have a similar predator--pest--host structure and can benefit from our analysis.

The mathematical model we use is built upon the ODE model of \cite{mooreborerhosseini}:
\begin{equation*}
    \begin{cases}
     I' & =\beta_{VH}(H-I)V_i-\alpha_II\\
     V_i' & =\beta_{HV}IV_s-d_VV_i-s_V(V_s+V_i)V_i-hPV_i\\
     V_s' & =-\beta_{HV}IV_s-d_VV_s-s_V(V_s+V_i)V_s-hPV_s+b_V(V_s+V_i)\\
     P' & =\gamma h(V_s+V_i)P-d_PP.
    \end{cases}
\end{equation*}
Here, the notations $I$, $V_i$, $V_s$, $P$ stand respectively for the population density of infected hosts, infected vectors, susceptible vectors and predators. The various parameters (all positive constants) are biologically meaningful: $H$ is the total population of hosts, $\beta_{VH}$ and $\beta_{HV}$ are the transmission rates (from vectors to hosts and vice-versa), $\alpha_I$ is the recovery rate of the hosts, $b_V$ is the birth rate of the vectors, $d_V$ and $d_P$ are
the death rates of the vectors and the predators respectively, $s_V$ is the saturation rate of vectors (intraspecific competition for space or resources leading to logistic growth), $h$ is the predation rate and $\gamma$ is a coefficient measuring the efficiency of predation.

To better account for the sugar beet agro-ecosystem, we modify this model as follows: 
\begin{itemize}
    \item since infected beets never recover, we remove the recovery term $-\alpha_I I$;
    \item on the contrary, aphids can recover, whence we add a recovery term with rate $\alpha$ for vectors;
    \item since we are interested in generalist predators, able to survive in the absence of aphids because of other prey populations present in the environment, we add a birth term with rate $b_P$ and a saturation term with rate $s_P$ so that predator growth in the absence of aphids is logistic;
    \item the field is an open bounded connected set $\Omega$ with smooth boundary $\partial\Omega$ in a Euclidean space $\mathbb{R}^d$;
    \item the spatial heterogeneity in the field due to the biodiversity refuges changes the Malthusian parameters $r_V=b_V-d_V$, $r_P=b_P-d_P$ and the total host population $H$ into positive functions in $\mathcal{C}^{2+\alpha}(\overline{\Omega})$;
    \item vectors and predators are able to move and disperse in this heterogeneous field: vectors use Brownian motion and diffuse with rate $\sigma_V$ whereas predators use a non-standard ``ideal free dispersal strategy'' \cite{cantrellcosnerlou2010,cantrellcosnerlou} with rate $\sigma_P$ that will be commented in more details below.
\end{itemize}

Our interest is therefore in the following parabolic--ordinary differential system:
\begin{equation}\label{eqspring}
 \begin{cases}
 \partial_tI & =\beta_{VH}(H(x)-I)V_i \\
 \partial_tV_i & =\sigma_V\Delta V_i+\beta_{HV}IV_s-\alpha V_i-d_V(x)V_i-s_V(V_s+V_i)V_i-hPV_i \\
 \partial_tV_s & =\sigma_V\Delta V_s-\beta_{HV}IV_s+\alpha V_i-d_V(x)V_s-s_V(V_s+V_i)V_s-hPV_s+b_V(x)(V_s+V_i) \\
 \partial_t P & =\sigma_P\nabla\cdot\left(r_P(x)\nabla\left(\dfrac{P}{r_P(x)}\right)\right)+\gamma h(V_s+V_i)P+r_P(x)P-s_PP^2
 \end{cases}
\end{equation}
in $\Omega\times\mathbb{R}_+^*$, where we ommited the dependency on $(x,t)$ for $I$, $V_i$, $V_s$ and $P$. We supplement it with initial conditions:
\[
(I, V_i, V_s, P)(\cdot,0) = (0, V_{i,0}, V_{s,0}, P_0)\quad\text{in }\Omega,
\] 
where $V_{i,0}$, $V_{s,0}$, $P_0$ are nonnegative functions in $L^{\infty}(\overline{\Omega})$, and with ``no-flux'' boundary conditions:
\[
\dfrac{\partial V_i}{\partial n}=\dfrac{\partial V_s}{\partial n}=\dfrac{\partial (P/r_P)}{\partial n}=0\quad\text{on }\partial\Omega.
\]

Note that the condition $\partial_n(P/r_P)=0$ reduces to $\partial_n P=0$ under the biologically relevant assumption $\partial_n r_P=0$, that will be a standing assumption from now on. Hence, we actually have homogeneous Neumann boundary conditions for all three animal populations $V_i$, $V_s$ and $P$.

At least formally, the ideal free dispersal strategy $\nabla\cdot\left(r_P\nabla\left(P/r_P\right)\right)$ can be rewritten as 
$\nabla\cdot\left(\nabla P-P\nabla(\ln r_P)\right)$. 
Diffusion--advection strategies of the form $u\mapsto\nabla\cdot(\nabla u - u\mathbf{q})$ with $u$ a population density subjected to a logistic growth of the form $u\mapsto u(r-u)$ with $r$ the spatially heterogeneous distribution of resources were compared in \cite{cantrellcosnerlou2010}. It was proved that
choosing an advection strategy of the form $\mathbf{q}=\nabla(\ln r)$ leads to a so-called ``ideal free distribution'' where the spatial distribution of the population density at equilibrium is exactly proportional to the spatial distribution of resources. In contrast, it is well-known that if $\mathbf{q}=\mathbf{0}$, that is if the population density is subjected to pure Laplacian 
diffusion, the equilibrium will not be proportional to the distribution of resources. The choice $\mathbf{q}=\nabla(\ln r_P)$, referred to as an ``ideal free dispersal strategy'', is both evolutionary stable and convergent stable \cite{cantrellcosnerlou2010}. Thus natural selection should favor such strategies, as conjectured by theoretical ecologists. Many other forms of ideal free dispersal strategies have been studied mathematically in the last few years. We emphasize in particular the case of patchy 
environments with edge behavior \cite{macielcosnercantrelllutscher} since this is the kind of environments we actually have in mind and will simulate numerically. In the present model, the ideal free distribution of the predators at equilibrium in the absence of aphids will be a key component of our attempt at an analysis of the optimal control problem. Indeed, the biodiversity refuges will act on the harvest through the population density $V_i$, which will itself be impacted by the refuges through $r_V$ on one hand and through the predator equilibrium on the other hand. Precise knowledge of this equilibrium is therefore very helpful, and actually quite necessary. We emphazise the fact that $r_P/s_P$ is an equilibrium only in the absence of aphids: to avoid complicated cross-diffusion terms, the ideal free distribution of the predators does not take into account the vector population $V_i+V_s$, only the other resources.

Since the predators we consider may typically have a broad view of the field, it is natural to assume that they will tend to optimize their spatial distribution based on resources. In contrast, vectors are aphids, and their movement is mostly governed by the wind, which could be approximated by a Brownian motion if we consider that, on average, the wind blows uniformly in all directions. Taking a Brownian motion for the predators would mean that their equilibrium state is no longer proportional to their resources, and would invalidate the optimization of the principal eigenvalue of Section \ref{spatialhomogenization}, but not the main theorem of this paper. One could also take an ideal free dispersal strategy for the vectors and the result of Theorem \ref{largetimebehaviour} would still be valid.

In the sugar beet agro-ecosystem, these equations correspond to the spring, when the beets are susceptible to the disease. It is known that in this system, the population dynamics are much faster than the seasonal flow, which means that by the end of spring, the system will have reached a steady state. This is consistent with our numerical simulations and with the parameter values we found in the literature (cf. Supplementary Materials). This is why we simplify by considering an eternal spring.

With this set of notations and assumptions, our optimal control problem can be recasted as follows: assuming that $r_V=r_V^{\text{field}}+r_V^{\text{refuge}}R$, $r_P=r_P^{\text{field}}+r_P^{\text{refuge}}R$, $H=H^{\text{field}}(1-R)$ with $R:\Omega\to[0,1]$ a spatial distribution of biodiversity refuges and with positive constants $r_V^{\text{field}}$, $r_V^{\text{refuge}}$, $r_P^{\text{field}}$, $r_P^{\text{refuge}}$, $H^{\text{field}}$, characterize the set of optimal $R$ such that the harvest of healthy beets $\lim_{t\to+\infty}\int_{\Omega} H-I(\cdot,t)$ is maximized. 

\subsection{Organization of the paper}

In the next subsection, we present our main result. In Section 2, we verify the well-posedness of the Cauchy--Neumann problem.
In Section 3, we prove our main result. In Section 4, we establish estimates on the harvest.
Finally, in Section 5, we study numerically the optimal control problem.

\subsection{Main result}

Let $\mathcal{L}_{V_s}=-\sigma_V\Delta-r_V+h\dfrac{r_P}{s_P}$, $\mathcal{L}_{V_i}=-\sigma_V\Delta+\alpha+d_V+h\dfrac{r_P}{s_P}$ and let $\lambda_1(\mathcal{L}_{V_s})$, $\lambda_1(\mathcal{L}_{V_i})$ be their principal eigenvalues with Neumann boundary conditions on $\partial\Omega$. $\lambda_1(\mathcal{L}_{V_i})$ is positive thanks to its Rayleigh quotient and the positivity of $\alpha+d_V+hr_P/s_P$ (see proof of Corollary \ref{eigvectequaleigscalar}); however, $\lambda_1(\mathcal{L}_{V_s})$ can be of any sign. Denote $\varphi_{s,1}$ and $\varphi_{i,1}$, the principal eigenfunctions associated, respectively, to $\lambda_1(\mathcal{L}_{V_s})$ and $\lambda_1(\mathcal{L}_{V_i})$, with normalizations $\min_{\Omega}\varphi_{s,1}=1$ and $\max_{\Omega}\varphi_{i,1}=1$.

Our main result is the following.

\begin{theorem}\label{largetimebehaviour}
 $(i)$ Case of extinction: assume that $P_0\not\equiv0$ and $\lambda_1(\mathcal{L}_{V_s})>0$. Let $V=V_i+V_s$. Then, the following statements are true.
 \begin{itemize}
  \item $(V,P)(x,t)\to (0,r_P(x)/s_P)$ as $t\to+\infty$, uniformly in $\Omega$. 
  \item If we also assume $V_{i,0}\not\equiv0$, then $\displaystyle\liminf_{t\to+\infty}\inf_{x\in\Omega}I(x,t)>0$.
  \item More precisely, we have the following estimates. For any $\varepsilon>0$ such that $\lambda_1(\mathcal{L}_{V_s})-h\varepsilon>0$, if $|P_0-r_P/s_P|\leq\varepsilon$ and $V_0:=V_{i,0}+V_{s,0}\leq\varepsilon$, then for all $x\in\Omega,\ t\geq 0$,
  \[
  V(x,t)\leq\max_{\Omega}V_0\max_{\Omega}\varphi_{s,1}e^{-(\lambda_1(\mathcal{L}_{V_s})-h\varepsilon)t},
  \]
  \[
  \frac{I(x,t)}{H(x)} \leq 1-\exp\left(-\beta_{VH}\frac{\max_{\Omega}V_0\max_{\Omega}\varphi_{s,1}}{\lambda_1(\mathcal{L}_{V_s})-h\varepsilon}\right),
  \]
  \[
  1-\exp\left(-\beta_{VH}\min_{\Omega}V_{i,0}\min_{\Omega}\varphi_{i,1}\dfrac{1-e^{-(\lambda_1(\mathcal{L}_{V_i})+s_V\varepsilon+h\varepsilon)t}}{\lambda_1(\mathcal{L}_{V_i})+s_V\varepsilon+h\varepsilon}\right)\leq\frac{I(x,t)}{H(x)}.
  \]
 \end{itemize}
 
 $(ii)$ Case of persistence: assume that $V_{i,0}\not\equiv0$ and $\lambda_1(\mathcal{L}_{V_s})<0$. Then,
 \begin{itemize}
     \item $\displaystyle\liminf_{t\to+\infty}\inf_{x\in\Omega}V_i(x,t)>0$;
     \item $\displaystyle\lim_{t\to+\infty}I(x,t)=H(x)$ uniformly in $\Omega$.
 \end{itemize}
\end{theorem}

This theorem, that does not require any particular relations between $r_V$, $r_P$, $H$ and the refuge distribution $R$, links the non-linear stability of the equilibrium $U_P=(0,0,0,r_P/s_P)^T$ with its linear stability. The proof combines linear stability arguments, supersolutions and subsolutions and a parabolic Harnack inequality in simultaneous time due to Huska \cite{huska}. In the persistence case, the control strategy fails and the whole field ends up being contaminated. On the contrary, in the extinction case, there might still be a salvageable harvest in the end. We point out that, although the persistence versus extinction criterion is classically the sign of some principal eigenvalue, the nonlinearities of the system still play a role in the extinction property.
In a purely linearly determined extinction scenario, the equilibrium $U_P$ would attract trajectories. This is for instance what happens classically in Fisher--KPP-type equations or systems. Here, in contrast, $(V_i,V_s,P)$ converges indeed to $(0,0,r_P/s_P)$ but $I$ ends up being at positive distance from $0$. We point out right now that the quantity of infected beets $I$ is nondecreasing in time. Moreover, when $V_i=V_s=0$ and $P=r_P/s_P$, any function $I_\infty$ of the spatial variable $x$ gives an equilibrium: the final quantity of infected beets cannot be figured out simply by calculating equilibria. 
Still, we provide estimates that bound the neighborhood of $U_P$ in which the evolution is trapped when initial conditions are well-prepared. The conditions $|P_0-r_P/s_P|\leq\varepsilon$ and $V_{i,0}+V_{s,0}\leq\varepsilon$ are biologically natural in a context of biological invasion by aphids in a territory where predators were previously close to equilibrium. Nevertheless, as will be shown by the proof, we can still provide estimates for ill-prepared initial conditions; these are just slightly more complicated.

As a by-product of these estimates, we can derive estimates for the harvest when initial conditions are well-prepared (i.e close to the equilibrium $(0,r_P(x)/s_P)$). We denote $I_\infty$ the pointwise limit of $I(t,x)$ as $t\to+\infty$ (which is well-defined since $I$ is monotonic with respect to time and globally bounded).

\begin{corollary}\label{estimatesharvest}
Assume that $\lambda_1(\mathcal{L}_{V_s})>0$. Then, for any $\varepsilon>0$ such that $\lambda_1(\mathcal{L}_{V_s})-h\varepsilon>0$, if $|P_0-r_P/s_P|\leq\varepsilon$ and $V_0\leq\varepsilon$, then
\[
    \mbox{e}^{-\beta_{VH}\dfrac{\max_{\Omega}V_0\max_{\Omega}\varphi_{s,1}}{\lambda_1(\mathcal{L}_{V_s})-h\varepsilon}}\leq\frac{\displaystyle\int_{\Omega}(H-I_\infty)}{\displaystyle\int_{\Omega}H}\leq\mbox{e}^{-\beta_{VH}\dfrac{\min_{\Omega}V_{i,0}\min_{\Omega}\varphi_{i,1}}{\lambda_1(\mathcal{L}_{V_i})+s_V\varepsilon+h\varepsilon}}.
\]

Assume that $\lambda_1(\mathcal{L}_{V_s})<0$. Then
\[
 \displaystyle\int_{\Omega}(H-I_\infty)=0.
\]
\end{corollary}

\section{Existence and unicity of the solution}

\begin{proposition}
 The system \eqref{eqspring} admits a unique classical solution defined in $\overline{\Omega}\times\mathbb{R}_+^*$.
\end{proposition}

\begin{proof}
  The local Lipschitz continuity of the functions involved in system \eqref{eqspring} and \cite[Page 447, Theorem 11.2]{pao} yield the existence of a unique classical solution in $\overline{\Omega}\times[0,T_0)$, for some $T_0\leq+\infty$. This solution either exists globally or blows up in finite time. It is left to show that the solution is bounded independently of $T_0$.
 
 First of all, because of the nonnegativity of the initial conditions $V_{i,0},V_{s,0}$ and $P_0$, $0$ is a trivial subsolution of each equation, hence the solutions $V_i$, $V_s$ and $P$ are nonnegative.
 
 $I$ can be computed explicitly, at least for $t\in[0,T_0)$:
 \[
  0\leq I(x,t)=H(x)\left(1-\exp\left(-\beta_{VH}\displaystyle\int_0^tV_i(x,s)\textnormal{d}s\right)\right).
 \]
 Since $\beta_{VH}>0$ and $V_i\geq 0$, it follows that $0\leq I(x, t)\leq H(x)$ for all $(x,t)\in\Omega\times[0,T_0)$: $I$ is bounded. Let us point out that this integral formula for $I(x,t)$ also confirms the time monotonicity and the fact that the final value of $I$ is not easily figured out, as it depends upon the whole trajectory of $V_i$, and in particular upon the initial condition $V_{i,0}$.
 
 Denote $V:=V_i+V_s$. The function $V$ solves
 \[
  \partial_tV-\sigma_V\Delta V = r_VV-s_VV^2-hPV.
 \]
 
 Let $\overline{V}$ be a solution of

 \begin{equation*}
     \begin{cases}
     \overline{V}(\cdot,0)=\underset{x\in\Omega}{\max}V_0:=\underset{x\in\Omega}{\max}(V_{i,0}+V_{s,0}) & \mbox{ in }\Omega\\
     \partial_t \overline{V}=\underset{x\in\Omega}{\max}(r_V)\overline{V}-s_V\overline{V}^2 & \mbox{ in }\Omega\times\mathbb{R}_+^*\\
     \dfrac{\partial\overline{V}}{\partial n}=0 & \mbox{ on }\partial\Omega\times\mathbb{R}_+^*.
     \end{cases}
 \end{equation*}
                                      
 By nonnegativity of $h,P,V$, $\overline{V}$ is a supersolution of the Cauchy--Neumann scalar problem of which $V$ is a solution. Moreover, $\overline{V}$ converges to $\max_{x\in\Omega}r_V(x)/s_V$ when $t\to+\infty$ (the solution can be computed explicitly), therefore is bounded. We then deduce the global boundedness of $V_i,V_s$: $0\leq V_i(x,t),V_s(x,t)\leq\overline{V}(x,t)$ for all $(x,t)\in\Omega\times[0,T_0)$, independently of $T_0$.
 
 Let $\mu\in\mathbb{R}$ be large enough so that $\overline{P}:=\mu r_P\geq P_0$ and $(r_P^2+\gamma h\overline{V}r_P)\mu-s_Pr_P^2\mu^2\leq0$. This is possible because $s_P,r_P$ are positive everywhere on the compact $\overline{\Omega}$. Then
 \[
 \begin{array}{ll}
  \partial_t \overline{P}-\sigma_P\nabla\cdot\left(r_P\nabla\left(\dfrac{\overline{P}}{r_P}\right)\right) = 0 & \geq (r_P^2+\gamma h\overline{V}r_P)\mu-s_Pr_P^2\mu^2\\
  & =\gamma hV\overline{P}+r_P\overline{P}-s_P\overline{P}^2,
 \end{array}
 \]
 hence $\overline{P}=\mu r_P$ is a supersolution of the Cauchy--Neumann scalar problem of which $P$ is a solution. Hence, $P$ is globally bounded, independently of $T_0$.

\end{proof}

\section{Stability of the predator-only equilibrium}

\subsection{Existence and unicity}

We are interested in studying the stability of the predator-only equilibrium $U_P:=(0,0,0,P^*)^T$, where $P^*$ is the nonnegative nonzero solution of
\begin{equation}\label{periodiceq}
 -\sigma_P\overline{\mathcal{L}}(P^*)=r_PP^*-s_P(P^*)^2\mbox{ in }\Omega.
\end{equation}
where we denote $\overline{\mathcal{L}}:=\overline{\mathcal{L}}(r_P):=\nabla\cdot\left(r_P\nabla\left(\dfrac{\cdot}{r_P}\right)\right)$.

\begin{proposition}\label{unicityofPstar}
 $P^*=\dfrac{r_P}{s_P}$ is the unique nonnegative nonzero solution of \eqref{periodiceq}.
\end{proposition}

\begin{proof}
 First of all, an easy computation shows that $r_P/s_P$ is a solution of \eqref{periodiceq}.
 
 According to \cite[Proposition 3.3]{cantrellcosner2004}, the nonnegative nonzero solution $P^*$ is unique if and only if the principal eigenvalue associated to the linearized operator around $0$, $\lambda_1(\mathcal{L}_0)$, is negative.
 
 We remind that $r_P$ is positive everywhere on $\overline{\Omega}$. Using the variational formulation of the elliptic principal eigenvalue $\lambda_1(\mathcal{L}_0)$,
 \[
 \begin{array}{ll}
  \lambda_1(\mathcal{L}_0) & =\inf_{\varphi\in\mathcal{C}^{2+\alpha}(\Omega),\Vert\varphi\Vert_{L^2}=1}\displaystyle\int_{\Omega}\sigma_Pr_P\nabla\left(\dfrac{\varphi}{r_P}\right)\cdot\nabla\varphi-r_P\varphi^2\\
  & \leq -\dfrac{\int_{\Omega}r_P^3}{\int_{\Omega}r_P^2}\quad\mbox{ (with }\varphi=r_P/\|r_P\|_{L^2})\\
  & <0.
 \end{array}
 \]
 
 This ends the proof.
\end{proof}

\subsection{Linear stability}

Denote $\varphi(x,t):=(I(x,t),V_i(x,t),V_s(x,t),P(x,t)-P^*(x))^T$.

From \eqref{eqspring}, a formal Taylor expansion at first order around $U_P=(0,0,0,r_P/s_P)^T$ yields
\[
\partial_t\varphi+\mathcal{L}\varphi=o(\Vert(I,V_i,V_s,P-P^*)\Vert)(1,1,1,1)^T
\]
where 
\[
\begin{array}{ll}
 \mathcal{L} & :=\begin{pmatrix}
\mathcal{L}_I & -\beta_{VH}H & 0 & 0\\
0 & \mathcal{L}_{V_i} & 0 & 0\\
0 & -\alpha-b_V & \mathcal{L}_{V_s} & 0\\
0 & -\gamma h\dfrac{r_P}{s_P} & -\gamma h\dfrac{r_P}{s_P} & \mathcal{L}_P
\end{pmatrix}\\
& :=\begin{pmatrix}
0 & -\beta_{VH}H & 0 & 0\\
0 & -\sigma_V\Delta+\alpha+d_V+h\dfrac{r_P}{s_P} & 0 & 0\\
0 & -\alpha-b_V & -\sigma_V\Delta-r_V+h\dfrac{r_P}{s_P} & 0\\
0 & -\gamma h\dfrac{r_P}{s_P} & -\gamma h\dfrac{r_P}{s_P} & -\sigma_P\overline{\mathcal{L}}+r_P
\end{pmatrix}.
\end{array}
\]

We look for solutions of the form $\varphi(x,t)=e^{-\lambda t}\psi(x)$, where $(\lambda,\psi)$ is an eigenpair of $\mathcal{L}$. 
Hence, we are focused on the linear elliptic eigenvalue problem
\[
 \mathcal{L}\psi=\lambda\psi.
\]

In the following, we will write $\operatorname{eig}(\mathcal{L})$ the set of eigenvalues of an operator $\mathcal{L}$, and $\sigma(\mathcal{L})$ its spectrum.

It is well known that any scalar elliptic operator $\mathcal{L}$ admits a unique principal eigenvalue. The Krein-Rutman theorem (see \cite{takac}) gives an associated positive principal eigenfunction to the principal eigenvalue.

\begin{definition}
 Let $\mathcal{L}$ be a scalar elliptic operator. The principal eigenvalue of $\mathcal{L}$, noted $\lambda_1(\mathcal{L})$, is the only eigenvalue associated to a positive eigenfunction.
\end{definition}

We also need a different definition that applies to more general operators.
\begin{definition}
 Let $\mathcal{L}$ be a linear operator with eigenvalues. We set
 \[
  \Lambda(\mathcal{L}):=\inf\;\operatorname{Re}(\operatorname{eig}(\mathcal{L})).
 \]
\end{definition}

We are interested in the stability of $U_P$. We will use the following terminology:

\begin{definition}\label{deflinearstability}
$U_P$ is :
\begin{itemize}
 \item linearly unstable if $\Lambda(\mathcal{L})<0$;
 \item linearly marginally stable if $\Lambda(\mathcal{L})=0$;
 \item linearly stable if $\Lambda(\mathcal{L})>0$.
\end{itemize}
\end{definition}

 The spectral problem $\mathcal{L}\psi=\lambda\psi$ is non-scalar, non-monotonous and non-compact (the equation on $I$ is not elliptic), hence difficult to study. In what follows, we take advantage of the specific form of the operator $\mathcal{L}$: a block triangular matrix, with each block being itself triangular.

We first need the following lemma.

\begin{lemma}
\label{spectrumsetofeigenvalues}
Let $\mathcal{E}$ a scalar elliptic operator with coefficients in $\mathcal{C}^{\alpha}(\overline{\Omega})$. Then, for any $K>0$ large enough, the spectrum of $\mathcal{E}+K$ is exactly the set of eigenvalues of $\mathcal{E}+K$.
\end{lemma}

\begin{proof}
We consider the problem $(\mathcal{E}+K)u=f$, for a given $f\in\mathcal{C}^{\alpha}(\overline{\Omega})$, where the solution $u$ belongs to $\mathcal{C}^{2+\alpha}(\overline{\Omega})$. Given Neumann boundary conditions, the existence and unicity theorem (see \cite[Page 447, Theorem 1.3]{pao}) shows that, for $K>0$ large enough, $\mathcal{E}+K:\mathcal{C}^{2+\alpha}(\overline{\Omega})\to\mathcal{C}^{\alpha}(\overline{\Omega})$ is bijective (and continuous). Furthermore, its inverse $(\mathcal{E}+K)^{-1}:\mathcal{C}^{\alpha}(\overline{\Omega})\to\mathcal{C}^{2+\alpha}(\overline{\Omega})$ is classically continuous. Hence, one can define $\iota:\mathcal{C}^{2+\alpha}(\overline{\Omega})\to\mathcal{C}^{\alpha}(\overline{\Omega})$ the canonical injection (which is compact, as a consequence of Arzela--Ascoli theorem), and $T:=\iota\circ(\mathcal{E}+K)^{-1}:\mathcal{C}^{\alpha}(\overline{\Omega})\to\mathcal{C}^{\alpha}(\overline{\Omega})$, so that $T$ is a compact operator.

Let $\lambda\in\sigma(\mathcal{E}+K)$ ($\lambda\neq0$). We claim the following:
\[
 \mathcal{E}+K-\lambda\mbox{ is invertible }\Leftrightarrow (\mathcal{E}+K)^{-1}-\dfrac{1}{\lambda}\mbox{ is invertible}.
\]

Indeed, suppose that $\mathcal{E}+K-\lambda$ is invertible. With $(\mathcal{E}+K)^{-1}$ invertible, their product $1-\lambda(\mathcal{E}+K)^{-1}=-\lambda((\mathcal{E}+K)^{-1}-1/\lambda)$ is also invertible. The proof of the converse is very similar.

We deduce that $1/\lambda\in\sigma((\mathcal{E}+K)^{-1})$.

If $1/\lambda\in\operatorname{eig}((\mathcal{E}+K)^{-1})$, then from $(\mathcal{E}+K)^{-1}u=(1/\lambda)u$ we deduce $\iota\circ(\mathcal{E}+K)^{-1}u=(1/\lambda)\iota\circ u=(1/\lambda)u$, hence $1/\lambda\in\sigma(T)$.

If $1/\lambda\notin\operatorname{eig}((\mathcal{E}+K)^{-1})$, then $(\mathcal{E}+K)^{-1}-1/\lambda$ is not surjective. Let $u\in\mathcal{C}^{2+\alpha}(\overline{\Omega})$ not attained by $(\mathcal{E}+K)^{-1}-1/\lambda$. Suppose we can find $f\in\mathcal{C}^{\alpha}(\overline{\Omega})$ such that $(T-1/\lambda)f=\iota\circ u$. Then $\iota\circ((\mathcal{E}+K)^{-1}-1/\lambda)f=(\iota\circ(\mathcal{E}+K)^{-1}-(1/\lambda)\iota)f=\iota\circ u$. By injectivity of $\iota$, $((\mathcal{E}+K)^{-1}-1/\lambda)f=u$, which contradicts our assumption on $u$. Therefore $T-1/\lambda$ is not surjective, and $1/\lambda\in\sigma(T)$.

In both cases, $1/\lambda\in\sigma(T)$.

By Fredholm's alternative, $1/\lambda\in\operatorname{eig}(T)$. By injectivity of $\iota$, $1/\lambda\in\operatorname{eig}((\mathcal{E}+K)^{-1})$. This shows that $\lambda\in\operatorname{eig}(\mathcal{E}+K)$, therefore $\sigma(\mathcal{E}+K)=\operatorname{eig}(\mathcal{E}+K)$.
\end{proof}

With this result, we are ready to assert the link between the eigenvalues of $\mathcal{L}$ and the eigenvalues of the other scalar operators.

\begin{proposition}\label{eigsetequalunion}
 $\operatorname{eig}(\mathcal{L})=\bigcup\limits_{Y\in\{I,V_i,V_s,P\}}\operatorname{eig}(\mathcal{L}_Y).$
\end{proposition}

\begin{proof}

Let $K>0$ so large that Lemma \ref{spectrumsetofeigenvalues} can be applied with $\mathcal{E}=\mathcal{L}_{V_i},\mathcal{L}_{V_s},\mathcal{L}_P$. Let $\lambda$ be an eigenvalue of $\mathcal{L}+K$ and $\varPhi(x)=(I(x),V_i(x),V_s(x),P(x))^T$ an eigenvector associated to it.

From the second line of $(\mathcal{L}+K)\varPhi(x)=\lambda\varPhi(x)$, we deduce
\[
(\mathcal{L}_{V_i}+K)V_i(x)=\lambda V_i(x),
\]
hence $\lambda$ is an eigenvalue of $\mathcal{L}_{V_i}+K$, or $V_i(x)=0$.

If $V_i(x)=0$, looking at the first line of $(\mathcal{L}+K)\varPhi(x)=\lambda\varPhi(x)$, we deduce
\[
(\mathcal{L}_I+K)I(x)=\lambda I(x),
\]
hence $\lambda$ is an eigenvalue of $\mathcal{L}_I+K$, or $I(x)=0$.

If $V_i(x)=0$ and $I(x)=0$, looking at the third line of $(\mathcal{L}+K)\varPhi(x)=\lambda\varPhi(x)$, we deduce
\[
(\mathcal{L}_{V_s}+K)V_s(x)=\lambda V_s(x),
\]
hence $\lambda$ is an eigenvalue of $\mathcal{L}_{V_s}+K$, or $V_s(x)=0$.

If $V_i(x)=0$, $I(x)=0$ and $V_s(x)=0$, looking at the fourth line of $(\mathcal{L}+K)\varPhi(x)=\lambda\varPhi(x)$, we deduce
\[
(\mathcal{L}_P+K)P(x)=\lambda P(x),
\]
hence $\lambda$ is an eigenvalue of $\mathcal{L}_P+K$, and $P$ cannot be null, otherwise $\varPhi$ would be null.

Thus, necessarily, $\lambda$ is an eigenvalue of one of the four operators $\mathcal{L}_I+K$, $\mathcal{L}_{V_i}+K$, $\mathcal{L}_{V_s}+K$, or $\mathcal{L}_P+K$.

Reciprocally, let us show that if $\lambda$ is an eigenvalue of one of the four operators $\mathcal{L}_I+K$, $\mathcal{L}_{V_i}+K$, $\mathcal{L}_{V_s}+K$, $\mathcal{L}_P+K$, then $\lambda$ is an eigenvalue of $\mathcal{L}+K$. By Lemma \ref{spectrumsetofeigenvalues}, for $K>0$ large enough and for all $Y\in\{V_i,V_s,P\}$, the spectrum of $\mathcal{L}_Y+K$ is exactly its set of eigenvalues. Also, $\mathcal{L}_I+K-\lambda=K-\lambda$ is trivially invertible for $\lambda\neq K$, hence $\operatorname{eig}(\mathcal{L}_I+K)=\sigma(\mathcal{L}_I+K)=\{ K\}$.

\textbf{First case:}

Let $\lambda$ be an eigenvalue of $\mathcal{L}_I+K$, and $\varphi(x)$ an eigenfunction associated to it.

Then we have $(\mathcal{L}+K)(\varphi(x),0,0,0)^T=\lambda(\varphi(x),0,0,0)^T$ and $\lambda$ is an eigenvalue of $\mathcal{L}+K$.

\textbf{Second case:}

Let $\lambda$ be an eigenvalue of $\mathcal{L}_{V_i}+K$, and $\psi(x)$ an eigenfunction associated to it. We are going to construct three functions $\varphi$, $\eta$ and $\mu$ such that

$(\mathcal{L}+K)(\varphi(x),\psi(x),\eta(x),\mu(x))^T=\lambda(\varphi(x),\psi(x),\eta(x),\mu(x))^T$.

First of all, as seen in the first case, if $\lambda\in\operatorname{eig}(\mathcal{L}_I+K)$, we have an obvious eigenvector showing that $\lambda$ is an eigenvalue of $\mathcal{L}+K$. Supposing now that $\lambda\notin\operatorname{eig}(\mathcal{L}_I+K)=\sigma(\mathcal{L}_I+K)$. $\mathcal{L}_I+K-\lambda$ is invertible, hence we can define $\varphi(x)=-(\mathcal{L}_I+K-\lambda)^{-1}\beta_{HV}H\psi(x)$.

If $\lambda\in\operatorname{eig}(\mathcal{L}_P+K)$, again we have a obvious eigenvector of $\mathcal{L}+K$: $(0,0,0,\mu(x))^T$ with $\mu$ an eigenfunction of $\mathcal{L}_P+K$ associated to $\lambda$. We suppose now that $\lambda\notin\operatorname{eig}(\mathcal{L}_P+K)=\sigma(\mathcal{L}_P+K)$.

If $\lambda\in\operatorname{eig}(\mathcal{L}_{V_s}+K)$, we pose $\eta$ an eigenfunction of $\mathcal{L}_{V_s}+K$ associated to it, and $\mu(x)=(\mathcal{L}_P+K-\lambda)^{-1}\gamma hP^*\eta(x)$. This way, we have $(\mathcal{L}+K)(0,0,\eta(x),\mu(x))^T=\lambda(0,0,\eta(x),\mu(x))^T$. We suppose now that $\lambda\notin\operatorname{eig}(\mathcal{L}_{V_s}+K)=\sigma(\mathcal{L}_{V_s}+K)$.

We define $\eta(x)=(\mathcal{L}_{V_s}+K-\lambda)^{-1}(\alpha+b_V)\psi(x)$ and

$\mu(x)=(\mathcal{L}_P+K-\lambda)^{-1}\gamma hP^*(\psi(x)+\eta(x))$. With $\varphi$, $\psi$, $\eta$ and $\mu$ defined as such, we have

$(\mathcal{L}+K)(\varphi(x),\psi(x),\eta(x),\mu(x))^T=\lambda(\varphi(x),\psi(x),\eta(x),\mu(x))^T$, thus $\lambda$ is an eigenvalue of $\mathcal{L}+K$.

\textbf{Third case:}

Let $\lambda$ be an eigenvalue of $\mathcal{L}_{V_s}+K$, and $\eta(x)$ an eigenfunction associated to it. If $\lambda\in\operatorname{eig}(\mathcal{L}_P+K)$, as seen before, we have an obvious eigenvector of $\mathcal{L}+K$. Suppose that $\lambda\notin\operatorname{eig}(\mathcal{L}_P+K)=\sigma(\mathcal{L}_P+K)$. We define $\mu(x)=(\mathcal{L}_P+K-\lambda)^{-1}\gamma hP^*\eta(x)$, and we have $(\mathcal{L}+K)(0,0,\eta(x),\mu(x))^T=\lambda(0,0,\eta(x),\mu(x))^T$, and $\lambda$ is an eigenvalue of $\mathcal{L}+K$.

\textbf{Fourth case:}

Let $\lambda$ be an eigenvalue of $\mathcal{L}_P+K$, and $\mu(x)$ an eigenfunction associated to it.
$(0,0,0,\mu(x))^T$ is an eigenvector of $\mathcal{L}+K$ associated to $\lambda$.

We showed that:
\[
 \operatorname{eig}(\mathcal{L}+K)=\bigcup\limits_{Y\in\{I,V_i,V_s,P\}}\operatorname{eig}(\mathcal{L}_Y+K).
\]

Therefore
\[
\begin{array}{ll}
 \operatorname{eig}(\mathcal{L}) & = \operatorname{eig}(\mathcal{L}+K)-K\\
 & = \bigcup\limits_{Y\in\{I,V_i,V_s,P\}}\operatorname{eig}(\mathcal{L}_Y+K)-K\\
 & = \bigcup\limits_{Y\in\{I,V_i,V_s,P\}}\operatorname{eig}(\mathcal{L}_Y).
\end{array}
\]

\end{proof}

With Proposition \ref{eigsetequalunion} in mind, the sign of the principal eigenvalues (in our case, the minimum of all eigenvalues) of the scalar operators actually yields a more precise result.

\begin{corollary}\label{eigvectequaleigscalar}
 $\Lambda(\mathcal{L})\in\operatorname{eig}(\mathcal{L})$ and $\Lambda(\mathcal{L)}=\min(0,\lambda_1(\mathcal{L}_{V_s}))$.
\end{corollary}

\begin{proof}

By Proposition \ref{eigsetequalunion}, and by definition of $\Lambda(\mathcal{L})$, $\Lambda(\mathcal{L})=\underset{Y\in\{I,V_i,V_s,P\}}{\min}\Lambda(\mathcal{L}_Y)$. Also, by Lemma \ref{spectrumsetofeigenvalues}, and by the fact that $\mathcal{L}_I=0$, for $K>0$ large enough, $\operatorname{eig}(\mathcal{L}_Y+K)=\sigma(\mathcal{L}_Y+K)$ for any $Y\in\{I,V_i,V_s,P\}$.

Hence, for all $Y\in\{V_i,V_s,P\}$, $\Lambda(\mathcal{L}_Y+K)=\lambda_1(\mathcal{L}_Y+K)$, which yields $\Lambda(\mathcal{L}_Y)=\lambda_1(\mathcal{L}_Y)$ . 
Moreover, as $-\lambda\mbox{Id}$ is not invertible if and only if $\lambda=0$, $\sigma(0)=\{0\}$, whence $\Lambda(\mathcal{L}_I)=0$.
Thus we only have to prove that $\lambda_1(\mathcal{L}_{V_i})>0$ and $\lambda_1(\mathcal{L}_P)>0$.

Using the variational formulation of the principal eigenvalue,
\[
 \lambda_1(\mathcal{L}_{V_i})=\inf_{\varphi\in\mathcal{C}^{2+\alpha}(\Omega),\Vert\varphi\Vert_{L^2}=1}\displaystyle\int_{\Omega}\sigma_V|\nabla\varphi|^2+(\alpha+d_V+h\dfrac{r_P}{s_P})\varphi^2>0.
\]

Finally, eigenelements $(\lambda,\varphi)$ of $\mathcal{L}_P$ verify
\[
 -\sigma_P\nabla\cdot\left(r_P\nabla\left(\dfrac{\varphi}{r_P}\right)\right)+r_P\varphi=\lambda\varphi,
\]
hence, multiplying the equation by $\varphi/r_P$ and integrating over $\Omega$,
\begin{center}
$\begin{array}{ll}
 \lambda\displaystyle\int_{\Omega}\dfrac{\varphi^2}{r_P} & =\displaystyle\int_{\Omega}-\sigma_P\nabla\cdot\left(r_P\nabla\left(\dfrac{\varphi}{r_P}\right)\right)\dfrac{\varphi}{r_P}+\varphi^2\\
 & =\displaystyle\int_{\Omega}\sigma_Pr_P\left|\nabla\dfrac{\varphi}{r_P}\right|^2+\varphi^2.
\end{array}$
\end{center}

By definition, the principal eigenfunction $\varphi_1$ of $\mathcal{L}_P$ is not identically $0$, therefore $\int_{\Omega}(\varphi_1^2/r_P)\neq0$. $\varphi_1$ being defined up to a multiplicative constant, we can choose it so that $\int_{\Omega}(\varphi_1^2/r_P)=1$. It yields $\lambda_1(\mathcal{L}_P)>0$.
\end{proof}

\subsection{Nonlinear stability}

We now link the linear stability with the nonlinear stability of $U_P$, or in other words we prove Theorem \ref{largetimebehaviour}. As Corollary \ref{eigvectequaleigscalar} shows, $U_P$ is marginally linearly stable if and only if $\lambda_1(\mathcal{L}_{V_s})\geq 0$. With the help of \cite[Proposition 3.2]{cantrellcosner2004} and a uniform Harnack inequality of \cite[Theorem 2.5]{huska}, we can make precise the large-time behaviour of our solution depending on the sign of $\lambda_1(\mathcal{L}_{V_s})$.

\begin{proof}[Proof of Theorem \ref{largetimebehaviour}]
 $(i)$ First of all, $0$ being an obvious subsolution of \eqref{eqspring}, we have $P,V_i,V_s\geq0$.
 
 $P$ and $V:=V_i+V_s$ are solutions of the system
 \begin{equation}\label{preypredsystem}
     \begin{cases}
     \partial_tV-\sigma_V\Delta V & =r_VV-s_VV^2-hPV\\
     \partial_t P-\sigma_P\overline{\mathcal{L}}(P) & =r_PP-s_PP^2+\gamma hVP,
     \end{cases}
 \end{equation}
 for $(x,t)\in\Omega\times\mathbb{R}_+^*$, with Neumann boundary conditions and initial conditions $P(\cdot,0)=P_0$, $V(\cdot,0)=V_0:=V_{i,0}+V_{s,0}$.
 
 Let $\underline{P}$ be a solution of
 \begin{equation*}
     \begin{cases}
      \underline{P}(\cdot,0)=P_0 & \mbox{ in }\Omega\\
      \partial_t \underline{P}-\sigma_P\overline{\mathcal{L}}(\underline{P})=r_P\underline{P}-s_P\underline{P}^2 & \mbox{ in }\Omega\times\mathbb{R}_+^*\\
      \dfrac{\partial\underline{P}}{\partial n}=0 & \mbox{ on }\partial\Omega\times\mathbb{R}_+^*.
     \end{cases}
 \end{equation*}
 
 By the nonnegativity of $\gamma$, $h$, $V$ and $\underline{P}$, $\underline{P}$ is a subsolution of the second equation of \eqref{preypredsystem}. Moreover, a linearization of the associated elliptic problem around the steady state $0$ yields the spectral problem
 \[
  (-\sigma_P\overline{\mathcal{L}}-r_P)\varphi=\lambda\varphi.
 \]
 
 By monotonicity of the principal eigenvalue with respect to the zeroth order term,
 \[
  \lambda_1(-\sigma_P\overline{\mathcal{L}}-r_P)\leq\lambda_1(-\sigma_P\overline{\mathcal{L}}-\underset{x\in\overline{\Omega}}{\min}r_P)\leq-\underset{x\in\overline{\Omega}}{\min}r_P<0.
 \]
 
 We used the fact that ($\lambda$,$\varphi$)=($-\underset{x\in\overline{\Omega}}{\min}r_P$,$r_P$) are eigenelements of the operator $-\sigma_P\overline{\mathcal{L}}-\underset{x\in\overline{\Omega}}{\min}r_P$.

 \cite[Proposition 3.2]{cantrellcosner2004} and Proposition \ref{unicityofPstar} yield the uniform convergence of $\underline{P}$ towards $\dfrac{r_P}{s_P}$.
 
 Let $\varepsilon>0$. There exists $t_0>0$ such that for all $x\in \Omega$ and $t\geq t_0$, $P(x,t)\geq\underline{P}(x,t)\geq\dfrac{r_P(x)}{s_P}-\varepsilon$. Denote $V_{t_0}:=V(\cdot,t_0)$, from the solution of \eqref{eqspring} taken at time $t=t_0$.
 
 Let $\overline{V}$ be a solution of
 \begin{equation*}
     \begin{cases}
      \overline{V}(\cdot,t_0)=V_{t_0} & \mbox{ in }\Omega\\
      \partial_t \overline{V}-\sigma_V\Delta\overline{V}=r_V\overline{V}-h(\dfrac{r_P}{s_P}-\varepsilon)\overline{V} & \mbox{ in }\Omega\times[t_0,+\infty)\\
      \dfrac{\partial\overline{V}}{\partial n}=0 & \mbox{ on }\partial\Omega\times[t_0,+\infty).
     \end{cases}
 \end{equation*}
 
 By our choice of $t_0$, $\overline{V}$ is a supersolution of the first equation of \eqref{preypredsystem} with $V_{t_0}$ as initial condition. Moreover, $\overline{V}$ is below $\underset{x\in\Omega}{\max}V(x,t_0)\underset{x\in\Omega}{\max}\varphi_{s,1}(x)e^{-(\lambda_1(\mathcal{L}_{V_s})-h\varepsilon)(t-t_0)}$, where $\varphi_{s,1}$ and $\lambda_1(\mathcal{L}_{V_s})$ satisfy
 \[
  (-\sigma_V\Delta+h(\dfrac{r_P}{s_P}-\varepsilon)-r_V)\varphi_{s,1}=(\lambda_1(\mathcal{L}_{V_s})-h\varepsilon)\varphi_{s,1}.
 \]
 with $\underset{x\in\Omega}{\min}\varphi_{s,1}(x)=1$.
 
 By positivity of $\lambda_1(\mathcal{L}_{V_s})$, one can choose $\varepsilon>0$ small enough so that $\lambda_1(\mathcal{L}_{V_s})-h\varepsilon>0$.
 
 Hence, one has the uniform convergence of
 
 $\overline{V}\leq\underset{x\in\Omega}{\max}V(x,t_0)\underset{x\in\Omega}{\max}\varphi_{s,1}(x)e^{-(\lambda_1(\mathcal{L}_{V_s})-h\varepsilon)(t-t_0)}$ to $0$ when $t\to+\infty$, yielding the uniform convergence of
 \[
 0\leq V_i,V_s\leq V\leq\overline{V}\leq\underset{x\in\Omega}{\max}V(x,t_0)\underset{\Omega}{\max}\varphi_{s,1}e^{-(\lambda_1(\mathcal{L}_{V_s})-h\varepsilon)(t-t_0)}
 \]
 to $0$. It is useful to point out that the uniform convergence of $V$ to $0$ demonstrates a control of the infection, and not the end of it.
 
 Now, take $t_0'>0$ such that for all $x\in\Omega$ and $t\geq t_0'$, $V(x,t)\leq\dfrac{\varepsilon}{\gamma h}\underset{x\in\overline{\Omega}}{\min}(r_P(x))$. Denote $P_{t_0'}:=P(\cdot,t_0')$, from the solution of \eqref{eqspring} taken at time $t=t_0'$.
 
 Let $\overline{P}$ be a solution of
 \begin{equation*}
     \begin{cases}
      \overline{P}(\cdot,t_0')=P_{t_0'} & \mbox{ in }\Omega\\
      \partial_t \overline{P}-\sigma_P\overline{\mathcal{L}}(\overline{P})=r_P\overline{P}-s_P\overline{P}^2+\varepsilon r_P\overline{P} & \mbox{ in }\Omega\times[t_0',+\infty)\\
      \dfrac{\partial\overline{P}}{\partial n}=0 & \mbox{ on }\partial\Omega\times[t_0',+\infty).
     \end{cases}
 \end{equation*}
 
 By our choice of $t_0'$, $\overline{P}$ is a supersolution of the second equation of \eqref{preypredsystem} with $P_{t_0'}$ as initial condition. Moreover, a linearization of the associated elliptic problem around the steady state $0$ yields the spectral problem
 \[
  (-\sigma_P\overline{\mathcal{L}}-(1+\varepsilon)r_P)\varphi=\lambda\varphi.
 \]
 
 By monotonicity of the principal eigenvalue with respect to the zeroth order term,
 \[
  \lambda_1(-\sigma_P\overline{\mathcal{L}}-(1+\varepsilon)r_P)\leq\lambda_1(-\sigma_P\overline{\mathcal{L}}-(1+\varepsilon)\underset{x\in\overline{\Omega}}{\min}r_P)\leq-(1+\varepsilon)\underset{x\in\overline{\Omega}}{\min}r_P<0.
 \]
 
 We used the fact that ($\lambda$,$\varphi$)=($-(1+\varepsilon)\underset{x\in\overline{\Omega}}{\min}r_P$,$r_P$) are eigenelements of the operator $-\sigma_P\overline{\mathcal{L}}-(1+\varepsilon)\underset{x\in\overline{\Omega}}{\min}r_P$.
 
  \cite[Theorem 3.2]{cantrellcosner2004} and a very similar version of Proposition \ref{unicityofPstar} (with $(1+\varepsilon)r_P$ instead of $r_P$) yield the uniform convergence of $\overline{P}$ towards $(1+\varepsilon)\dfrac{r_P}{s_P}$. Combining the convergence of $\underline{P}$ and $\overline{P}$, there exist a $t_1>0$ large enough so that, for all $x\in\Omega$ and $t\geq t_1$, $\dfrac{r_P(x)}{s_P}-\varepsilon\leq\underline{P}(x,t)\leq P(x,t)\leq\overline{P}\leq (1+\varepsilon)\dfrac{r_P(x)}{s_P}+\varepsilon$.
 
 Since $\varepsilon$ is arbitrary, this shows the uniform convergence of $P$ to $r_P/s_P$.
 
 Finally, assume $V_{i,0}\not\equiv 0$ and let $\delta>0$. Suppose by contradiction that $\underset{x\in\Omega}{\min}V_i(x,\delta)=0$. By the uniform Harnack inequality (see \cite[Theorem 2.5]{huska}), there exists $C_{\delta}>0$ such that $\underset{x\in\Omega}{\max}V_i(x,\delta)\leq C_{\delta}\underset{x\in\Omega}{\min}V_i(x,\delta)=0$ for all $x\in\Omega$, therefore we would have $V_i(x,\delta)=0$ for all $x\in\Omega$. By uniqueness of the Cauchy--Neumann problem, $V$ would be identically $0$ for all time, which contradicts the fact that $V_{i,0}\not\equiv0$. Therefore, for all $x\in\Omega$,
 \[
  I(x,\delta)=H(x)\left(1-\exp\left(-\beta_{VH}\displaystyle\int_0^{\delta}V_i(x,s)\mbox{d}s\right)\right)>0.
 \]
 
 By the Hopf boundary lemma and the homogeneous Neumann boundary conditions, for all $t\in(0,\delta)$, $\underset{x\in\Omega}{\inf}\;V_i(x,t)>0$. Then, for all $x\in\Omega$, $\displaystyle\int_0^{\delta}V_i(x,s)\mbox{d}s\geq\displaystyle\int_0^{\delta}\underset{x\in\Omega}{\inf}\;V_i(x,s)\mbox{d}s$, which yields $\underset{x\in\Omega}{\inf}\displaystyle\int_0^{\delta}V_i(x,s)\mbox{d}s>0$. By positivity of $H$ on the compact $\overline{\Omega}$, $\underset{x\in\Omega}{\inf}\;I(x,\delta)>0$. By increasingness of $t\mapsto I(x,t)$,
 \[
  \underset{t\to+\infty}{\liminf}\;\underset{x\in\Omega}{\inf}\;I(x,t)\geq\underset{x\in\Omega}{\inf}\;I(x,\delta)>0.
 \]
 
 Let $t_{\varepsilon}>0$ such that, for all $x\in\Omega$ and $t\geq t_{\varepsilon}$, $V(x,t)\leq\varepsilon$ and $|P(x,t)-r_P/s_P|\leq\varepsilon$. Denote $\underline{V_i}$, the solution of
 \begin{equation*}
     \begin{cases}
      \underline{V_i}(x,t_{\varepsilon})=V_{i,t_{\varepsilon}} & \mbox{ in }\Omega\\
      \partial_t\underline{V_i}-\sigma_V\Delta\underline{V_i}=(-\alpha-d_V-s_V\varepsilon-h(r_P/s_P+\varepsilon))\underline{V_i} & \mbox{ in }\Omega\times[t_{\varepsilon},+\infty)\\
      \dfrac{\partial\underline{V_i}}{\partial n}=0 & \mbox{ on }\partial\Omega\times[t_{\varepsilon},+\infty)
     \end{cases}
 \end{equation*}
 
 Denote $\varphi_{i,1}$ the principal eigenfunction associated to $\lambda_1(\mathcal{L}_{V_i})$ (which is positive, by the Krein-Rutman theorem), with $\underset{x\in\Omega}{\max}\varphi_{i,1}(x)=1$. By our choice of $t_{\varepsilon}$, $\underline{V_i}$ is a subsolution of the Cauchy--Neumann scalar system verified by $V_i$, and for all $x\in\Omega$, $t\geq t_{\varepsilon}$,
 \[
 V_i(x,t)\geq\underset{x\in\Omega}{\min}V_i(x,t_{\varepsilon})\underset{\Omega}{\min}\varphi_{i,1}\exp\left(-(\lambda_1(\mathcal{L}_{V_i})+s_V\varepsilon+h\varepsilon)(t-t_{\varepsilon})\right).
 \]
 
 By the supersolution of $V_i$ we have exibited before, for all $x\in\Omega,\ t\geq t_{\varepsilon}$,
 \[
 \begin{array}{ll}
    & I(x,t)/H(x)\\
    = & 1-\exp\left(-\beta_{VH}\displaystyle\int_0^tV_i(x,s)\mbox{d}s\right)\\
    \leq & 1-\exp\left(-\beta_{VH}\left(\displaystyle\int_0^{t_{\varepsilon}}V_i(x,s)\mbox{d}s+\underset{x\in\Omega}{\max}V(x,t_{\varepsilon})\underset{\Omega}{\max}\varphi_{s,1}\displaystyle\int_{t_{\varepsilon}}^te^{-(\lambda_1(\mathcal{L}_{V_s})-h\varepsilon)(s-t_{\varepsilon})}\mbox{d}s\right)\right)\\
    \leq & 1-\exp\left(-\beta_{VH}\left(\displaystyle\int_0^{t_{\varepsilon}}V_i(x,s)\mbox{d}s+\dfrac{\underset{x\in\Omega}{\max}V(x,t_{\varepsilon})\underset{\Omega}{\max}\varphi_{s,1}}{\lambda_1(\mathcal{L}_{V_s})-h\varepsilon}\right)\right).
 \end{array}
 \]

 By the subsolution of $V_i$ we have exibited before, for all $x\in\Omega,\ t\geq t_{\varepsilon}$,
 \[
 \begin{array}{ll}
    & I(x,t)/H(x)\\
    \geq & 1-\exp\left(-\beta_{VH}\left(\displaystyle\int_0^{t_{\varepsilon}}V_i(x,s)\mbox{d}s+\underset{x\in\Omega}{\min}V_i(x,t_{\varepsilon})\underset{\Omega}{\min}\varphi_{i,1}\displaystyle\int_{t_{\varepsilon}}^te^{-(\lambda_1(\mathcal{L}_{V_i})+s_V\varepsilon+h\varepsilon)(s-t_{\varepsilon})}\mbox{d}s\right)\right)\\
    = & 1-\exp\left(-\beta_{VH}\left(\displaystyle\int_0^{t_{\varepsilon}}V_i(x,s)\mbox{d}s+\underset{x\in\Omega}{\min}V_i(x,t_{\varepsilon})\underset{\Omega}{\min}\varphi_{i,1}\dfrac{1-e^{-(\lambda_1(\mathcal{L}_{V_i})+s_V\varepsilon+h\varepsilon)(t-t_{\varepsilon})}}{\lambda_1(\mathcal{L}_{V_i})+s_V\varepsilon+h\varepsilon}\right)\right).
 \end{array}
 \]
 
 If $|P_0-r_P/s_P|\leq\varepsilon$ and $V_0\leq\varepsilon$, then $t_{\varepsilon}=0$ and the proof is completed.
 
 $(ii)$ Let us first prove that
 \[
  \underset{t\to+\infty}{\liminf}\;\underset{x\in\Omega}{\inf}\;V(x,t)>0.
 \]
 
 We are going to need the following lemma:
 
 \begin{lemma}
  Let $\varepsilon>0$, and let $[t_0,T)$ be any interval of time where $V$ is lower than or equal to $\varepsilon$. There exists $S_{\varepsilon}>0$ such that $T\leq S_{\varepsilon}$.
 \end{lemma}
 
 \begin{proof}
 Let $\varepsilon>0$ be fixed. During the course of the proof, we will make a smallness assumption on its value that could have been made from the start without loss of generality. Let $[t_0,T)$ be an interval of time where $V(x,t)\leq\varepsilon$ for all $x\in\Omega$. We are going to prove that $T\leq S_{\varepsilon}$, and we are going to give an explicit formula for $S_{\varepsilon}$. Let us suppose, by contradiction, that $T=+\infty$.
 
 We can consider $\overline{P}$, the solution of the Cauchy scalar problem
 \begin{equation*}
     \begin{cases}
      \overline{P}(t_0)=m_{P_0}:=\underset{x\in\Omega}{\max}\left(\dfrac{P_{t_0}}{r_P}\right)\\
      \partial_t \overline{P}=\underset{x\in\Omega}{\max}(r_P)(\overline{P}-s_P\overline{P}^2)+\gamma h\varepsilon\overline{P} & \mbox{ in }[t_0,T).
     \end{cases}
 \end{equation*}
 
 One easily checks that $\overline{P}$ is a supersolution of
 \begin{equation*}
     \begin{cases}
      \tilde{P}(.,t_0)=\dfrac{P_{t_0}}{r_P} & \mbox{ in }\Omega\\
      \partial_t\tilde{P}-\dfrac{\sigma_P}{r_P}\nabla\cdot\left(r_P\nabla(\tilde{P})\right)=r_P(\tilde{P}-s_P\tilde{P}^2)+\gamma hV\tilde{P} & \mbox{ in }\Omega\times[t_0,T)\\
      \dfrac{\partial\tilde{P}}{\partial n}=0 & \mbox{ on }\partial\Omega\times[t_0,T).
     \end{cases}
 \end{equation*}

 Hence, $P=r_P\tilde{P}\leq r_P\overline{P}$ on this interval. Moreover, one can compute explicitly
 \[
  \overline{P}(t)=K_{\varepsilon}\dfrac{1}{1+\left(K_{\varepsilon}/m_{P_0}-1\right)e^{-(\underset{x\in\Omega}{\max}(r_P)+\gamma h\varepsilon)(t-t_0)}},
 \]
 which converges to $K_{\varepsilon}:=\dfrac{\underset{x\in\Omega}{\max}(r_P)+\gamma h\varepsilon}{\underset{x\in\Omega}{\max}(r_P)s_P}$. Hence, there exists $t_1\geq t_0$ such that
 \[
  \overline{P}(t)\leq K_{\varepsilon}+\varepsilon
 \]
 for all $t\geq t_1$. Assume $t_1$ is minimal.
 
 If $m_{P_0}s_P\leq1$, then $\overline{P}$ is increasing in time and $\overline{P}(t)\leq K_{\varepsilon}$ for all $t\geq t_0$, hence $t_1=t_0$.
 
 If $m_{P_0}s_P>1$, then $\overline{P}$ is decreasing in time, $\overline{P}(t)\leq m_{P_0}$ for all $t\geq t_0$, and $t_1$ is given by
 \[
  \overline{P}(t_1)=K_{\varepsilon}+\varepsilon,
 \]
 hence
 \[
  e^{-(\underset{x\in\Omega}{\max}(r_P)+\gamma h\varepsilon)(t-t_0)}=\dfrac{1}{(1+K_{\varepsilon}/\varepsilon)\left(1-K_{\varepsilon}/m_{P_0}\right)},
 \]
 and
 \[
  t_1=t_0+\dfrac{\ln\left(1+K_{\varepsilon}/\varepsilon\right)+\ln\left(1-K_{\varepsilon}/m_{P_0}\right)}{\underset{x\in\Omega}{\max}(r_P)+\gamma h\varepsilon}.
 \]
 
 Denote $\overline{\mathcal{P}}:=\underset{x\in\Omega}{\max}(r_P)\max\left(K_{\varepsilon},m_{P_0}\right)$, so that, for all $x\in\Omega$ and $t\geq t_0$, $P(x,t)\leq\overline{\mathcal{P}}$. We can construct a subsolution $\underline{U}$ of the Cauchy--Neumann scalar problem verified by $V$ on $[t_0,T)$:
 \begin{equation*}
     \begin{cases}
      \underline{U}(t_0)=\min\left(\underset{x\in\Omega}{\min}V_{t_0},\dfrac{\underset{x\in\Omega}{\min}r_V}{s_V}\right)\\
      \partial_t \underline{U}=-h\overline{\mathcal{P}}\underline{U} & \mbox{ in }[t_0,T),
     \end{cases}
 \end{equation*}
 which gives us, for all $x\in\Omega$, the following inequality:
  \[
  V(x,t_1)\geq C:=\min\left(\underset{x\in\Omega}{\min}V_{t_0},\dfrac{\underset{x\in\Omega}{\min}r_V}{s_V}\right)e^{-h\overline{\mathcal{P}}(t_1-t_0)}.
 \]

 Now, let $\underline{V}$ be the solution of
 \begin{equation*}
     \begin{cases}
      \underline{V}(\cdot,t_1)=V_{t_1} & \mbox{ in }\Omega\\
      \partial_t \underline{V}-\sigma_V\Delta\underline{V}=r_V\underline{V}-2s_V\varepsilon\underline{V}-hr_P\left(K_{\varepsilon}+\varepsilon\right)\underline{V} & \mbox{ in }\Omega\times[t_1,T)\\
      \dfrac{\partial\underline{V}}{\partial n}=0 & \mbox{ on }\partial\Omega\times[t_1,T).
     \end{cases}
 \end{equation*}
 
 By construction of $t_1$, $\underline{V}$ is a subsolution of the Cauchy--Neumann scalar problem verified by $V$ on this interval, as long as $\underline{V}\leq2\varepsilon$. Say that $\underline{V}\leq2\varepsilon$ in an interval $[t_1,t^*]$ ($t^*>t_1$ exists by the assumption $V\leq\varepsilon$ and by continuity of $\underline{V}$). By negativity of $\lambda_1(\mathcal{L}_{V_s})$ and by continuity of the principal eigenvalue with respect to the zeroth order term, one can take $\varepsilon$ small enough such that
 \[
  \lambda_1^{\varepsilon}:=\lambda_1\left(-\sigma_V\Delta-r_V+2s_V\varepsilon+hr_P\left(K_{\varepsilon}+\varepsilon\right)\right)<0
 \]
 (recall that $K_{\varepsilon}\to1/s_P$ as $\varepsilon\to0$).
 
 Even if it means reducing it, take $\varepsilon$ satisfying this condition.
 
 The subsolution $\underline{V}$ is above $C\varphi_1^{\varepsilon}(x)e^{-\lambda_1^{\varepsilon}(t-t_1)}$ for all $x\in\Omega$ and $t\geq t_1$, $\varphi_1^{\varepsilon}$ being the principal eigenvalue associated to $\lambda_1^{\varepsilon}$ (which is positive, by the Krein-Rutman theorem) with $\underset{x\in\Omega}{\max}\;\varphi_1^{\varepsilon}=1$. Therefore, for all $x\in\Omega$ and $t\in [t_1,t^*]$,
 \[
 \begin{array}{ll}
  V(x,t) & \geq\underline{V}(x,t)\\
  & \geq C\varphi_1^{\varepsilon}(x)e^{-\lambda_1^{\varepsilon}(t-t_1)}\\
  & \geq \min\left(\underset{x\in\Omega}{\min}V_{t_0},\dfrac{\underset{x\in\Omega}{\min}r_V}{s_V}\right)e^{-h\overline{\mathcal{P}}(t_1-t_0)}\underset{\Omega}{\min}\varphi_1^{\varepsilon}e^{-\lambda_1^{\varepsilon}(t-t_1)},
 \end{array}
 \]
 hence $\underline{V}$ is above an increasing function until the finite time $t^*$ where $\underline{V}(x,t^*)=2\varepsilon$ for all $x\in\Omega$. According to the previous inequality, one has, for all $x\in\Omega$, $V(x,t^*)\geq2\varepsilon>\varepsilon$, which contradicts the assumption on $T$. $T$ is hereby finite, and we can compute $S_{\varepsilon}$:
 \[
  \min\left(\underset{x\in\Omega}{\min}V_{t_0},\dfrac{\underset{x\in\Omega}{\min}r_V}{s_V}\right)e^{-h\overline{\mathcal{P}}(t_1-t_0)}\underset{\Omega}{\min}\varphi_1^{\varepsilon}e^{-\lambda_1^{\varepsilon}(t_0+S_{\varepsilon}-t_1)}=\varepsilon,
 \]
 hence
 \[
  S_{\varepsilon}=\dfrac{\ln\left(\min\left(\underset{x\in\Omega}{\min}V_{t_0},\dfrac{\underset{x\in\Omega}{\min}r_V}{s_V}\right)e^{-h\overline{\mathcal{P}}(t_1-t_0)}\underset{\Omega}{\min}\varphi_1^{\varepsilon}\right)-\ln(\varepsilon)}{\lambda_1^{\varepsilon}}+t_1-t_0,
 \]
 with the value of $t_1-t_0$ given above.
 \end{proof}
 
 First of all, let us prove that, for any $\delta>0$, $\underset{x\in\Omega}{\min}V(x,\delta)>0$. Let $\delta>0$. Suppose by contradiction that $\underset{x\in\Omega}{\min}V(x,\delta)=0$. By the uniform Harnack inequality (see \cite[Theorem 2.5]{huska}), there exists $C_{\delta}>0$ such that $\underset{x\in\Omega}{\max}V(x,\delta)\leq C_{\delta}\underset{x\in\Omega}{\min}V(x,\delta)=0$ for all $x\in\Omega$, therefore we would have $V(x,\delta)=0$ for all $x\in\Omega$. By unicity of the Cauchy--Neumann problem, $V$ would be identically $0$ for all time, which contradicts the fact that $V_0\not\equiv0$.
 
 Let $\varepsilon>0$. We are going to consider two cases:
 
 \textbf{Case 1:} there exists $t^*>0$ such that, for all $t\geq t^*$, $\underset{x\in\Omega}{\max}V(x,t)\geq\varepsilon$. By the uniform Harnack inequality (see \cite[Theorem 2.5]{huska}), there exists $C_{t^*}>0$ so that, for all $t\geq t^*$, $\underset{x\in\Omega}{\min}V(x,t)\geq C_{t^*}\underset{x\in\Omega}{\max}V(x,t)\geq C_{t^*}\varepsilon$. The proof is then completed.
 
 \textbf{Case 2:} for all $t^*>0$, there exists $t\geq t^*$ such that $\underset{x\in\Omega}{\max}V(x,t)<\varepsilon$. By the lemma, we can deduce the existence of an increasing sequence of times $(t_n)_{n\in\mathbb{N}}$, $t_0>0$, $t_n\underset{n\to+\infty}{\longrightarrow}+\infty$, where, for all $n\in\mathbb{N}$,
 \[
  \begin{cases}
   0<\underset{x\in\Omega}{\max}V(x,t)\leq\varepsilon & \forall t\in[t_{2n},t_{2n+1}]\\
   \underset{x\in\Omega}{\max}V(x,t)\geq\varepsilon & \forall t\in[t_{2n+1},t_{2n+2}].
  \end{cases}
 \]
 
 By the uniform Harnack inequality (see \cite[Theorem 2.5]{huska}), there exists $C_{t_0}>0$ so that, for all $t\geq t_0$, $\underset{x\in\Omega}{\min}V(x,t)\geq C_{t_0}\underset{x\in\Omega}{\max}V(x,t)$.
 
 Hence, for all $n\in\mathbb{N}$ and all $t\in[t_{2n+1},t_{2n+2}]$, $\underset{x\in\Omega}{\min}V(x,t)\geq C_{t_0}\underset{x\in\Omega}{\max}V(x,t)\geq C_{t_0}\varepsilon$.

 Again by the lemma, we know that, for all $n\in\mathbb{N}$, $|t_{2n+1}-t_{2n}|\leq S_{\varepsilon}$. Therefore, for all $n\in\mathbb{N}$ and all $t\in[t_{2n},t_{2n+1}]$,
 \[
   \underset{x\in\Omega}{\min}V(x,t)\geq\underset{t\in[t_{2n},t_{2n}+S_{\varepsilon}]}{\min}\underset{x\in\Omega}{\min}V(x,t)
 \]
 
 Denote $\overline{\mathcal{P}}:=\underset{x\in\Omega}{\max}(r_P)\max\left(\dfrac{\underset{x\in\Omega}{\max}(r_P)+\gamma h\varepsilon}{\underset{x\in\Omega}{\max}(r_P)s_P},\underset{x\in\Omega}{\max}\left(\dfrac{P_{t_0}}{r_P}\right)\right)$. Without detailing more (see the proof of the lemma), we have for all $x\in\Omega$ and $t\geq t_0$, $P(x,t)\leq\overline{\mathcal{P}}$. We can construct, for all $n\in\mathbb{N}$, a subsolution $\underline{U}$ of the Cauchy--Neumann scalar problem verified by $V$ on $[t_{2n},t_{2n}+S_{\varepsilon}]$:
 \[
 \begin{cases}
 \underline{U}(t_0)=C_{t_0}\varepsilon\\
 \partial_t \underline{U}=-h\overline{\mathcal{P}}\underline{U} & \mbox{ in }[t_{2n},t_{2n}+S_{\varepsilon}],
 \end{cases}
 \]
 which yields, for all $n\in\mathbb{N}$ and all $t\in[t_{2n},t_{2n}+S_{\varepsilon}]$,
 \[
  \underset{x\in\Omega}{\min}V(x,t)\geq C_{t_0}\varepsilon e^{-h\overline{\mathcal{P}}S_{\varepsilon}},
 \]
 hence
 \[
  \underset{t\in[t_{2n},t_{2n}+S_{\varepsilon}]}{\min}\underset{x\in\Omega}{\min}V(x,t)\geq C_{t_0}\varepsilon e^{-h\overline{\mathcal{P}}S_{\varepsilon}}.
 \]
 
 Finally, combining the inequalities on all segments, for all $t\geq t_0$,
 \[
  \underset{x\in\Omega}{\min}V(x,t)\geq\underline{\mathcal{V}}:=\min\left(C_{t_0}\varepsilon,C_{t_0}\varepsilon e^{-h\overline{\mathcal{P}}S_{\varepsilon}}\right),
 \]
 which proves that
 \[
  \underset{t\to+\infty}{\liminf}\;\underset{x\in\Omega}{\inf}\;V(x,t)>0.
 \]
 
 $V_i$ satisfies
 \begin{equation*}
     \begin{cases}
      V_i(\cdot,0)=V_{i,0}\not\equiv 0 & \mbox{ in }\Omega\\
      \partial_t V_i-\sigma_V\Delta V_i=\beta_{HV}IV_s-\alpha V_i-d_VV_i-s_V(V_s+V_i)V_i-hPV_i & \mbox{ in }\Omega\times\mathbb{R}_+^*\\
      \dfrac{\partial V_i}{\partial n}=0 & \mbox{ on }\partial\Omega\times\mathbb{R}_+^*.
     \end{cases}
 \end{equation*}
 
 By classical arguments presented before, for all $\delta>0$, $t\geq\delta$ and $x\in\Omega$, $V_i(x,t)>0$. In particular, $\underset{x\in\Omega}{\min}V_i(x,t_0)>0$.
 
 The equation on $V_i$ can be rewritten as
 \[
  \begin{array}{ll}
    \partial_t V_i-\sigma_V\Delta V_i & =\beta_{HV}I(V-V_i)-\alpha V_i-d_VV_i-s_VVV_i-hPV_i\\
    & =\beta_{HV}H\left(1-\exp\left(-\beta_{VH}\displaystyle\int_0^tV_i\right)\right)V-(\beta_{HV}I+\alpha+d_V+s_VV+hP)V_i.
  \end{array}
 \]
 
 $I$, $V$ and $P$ being bounded, there exists a constant $C>0$ large enough so that $\beta_{HV}I+\alpha+d_V+s_VV+hP\leq C$. Define $\underline{V_i}$, the solution of
 \begin{equation*}
     \begin{cases}
      \underline{V_i}(t_0)=\underset{x\in\Omega}{\min}V_i(x,t_0)\\
      \partial_t\underline{V_i}=\beta_{HV}\underset{x\in\Omega}{\min}H(x)\left(1-\exp\left(-\beta_{VH}\displaystyle\int_0^{t_0}\underset{x\in\Omega}{\min}V_i(x,s)\mbox{d}s\right)\right)\underline{\mathcal{V}}-C\underline{V_i} & \mbox{ in }[t_0,+\infty),
     \end{cases}
 \end{equation*}
 
 One easily checks that $\underline{V_i}$ is a subsolution of the Cauchy--Neumann scalar system verified by $V_i$, and has
 \[
  \underline{V_i}(t)=a+(\underset{x\in\Omega}{\min}V_i(x,t_0)-a)e^{-C(t-t_0)},
 \]
 where
 \[
 a=\dfrac{1}{C}\beta_{HV}\underset{x\in\Omega}{\min}H(x)\left(1-\exp\left(-\beta_{VH}\displaystyle\int_0^{t_0}\underset{x\in\Omega}{\min}V_i(x,s)\mbox{d}s\right)\right)\underline{\mathcal{V}}>0,
 \]
 hence
 \[
  \underset{t\to+\infty}{\liminf}\;\underset{x\in\Omega}{\inf}\;V_i(x,t)\geq\underset{t\to+\infty}{\lim}\;\underline{V_i}(t)=a>0.
 \]
 
 Moreover, for all $x\in\Omega$ and $t\geq t_0$,
 \[
 \begin{array}{ll}
    |I(x,t)-H(x)| & =H(x)\exp\left(-\beta_{VH}\displaystyle\int_0^tV_i(x,s)\mbox{d}s\right)\\
    & \leq H(x)\exp\left(-\beta_{VH}\left(\displaystyle\int_0^{t_0}V_i(x,s)\mbox{d}s+\displaystyle\int_{t_0}^ta+\displaystyle\int_{t_0}^t(\underset{x\in\Omega}{\min}V_i(x,t_0)-a)e^{-C(s-t_0)}\mbox{d}s\right)\right)\\
    & =H(x)\exp\left(-\beta_{VH}\left(\displaystyle\int_0^{t_0}V_i(x,s)\mbox{d}s+(t-t_0)a+(\underset{x\in\Omega}{\min}V_i(x,t_0)-a)\dfrac{1-e^{-C(t-t_0)}}{C}\right)\right).
 \end{array}
 \]
 
 By boundedness of $H$ and $V_i$, $H\exp\left(-\beta_{VH}\left(\displaystyle\int_0^{t_0}V_i(\cdot,s)\mbox{d}s\right)\right)$ is bounded, say by a constant $C'>0$. Therefore, for all $t\geq t_0$,
 \[
 \begin{array}{ll}
    \Vert I(\cdot,t)-H\Vert_{\infty} & \leq C'\exp\left(-\beta_{VH}\left((t-t_0)a+(\underset{x\in\Omega}{\min}V_i(x,t_0)-a)\dfrac{1-e^{-C(t-t_0)}}{C}\right)\right)\\
    & \underset{t\to+\infty}{\longrightarrow}0.
 \end{array}
 \]
\end{proof}

\section{Estimations of the harvest}

Assume that $\lambda_1(\mathcal{L}_{V_s})>0$ and fix $\varepsilon>0$ such that $\lambda_1(\mathcal{L}_{V_s})-h\varepsilon>0$. By theorem \ref{largetimebehaviour}, there exists $t_{\varepsilon}>0$ such that, for all $x\in\Omega$, $t\geq t_{\varepsilon}$,
\[
 \begin{array}{ll}
    & H(x)\left(1-\exp\left(-\beta_{VH}\left(\displaystyle\int_0^{t_{\varepsilon}}V_i(x,s)\mbox{d}s+\underset{x\in\Omega}{\min}V_{i}(x,t_{\varepsilon})\underset{\Omega}{\min}\varphi_{i,1}\dfrac{1-e^{-(\lambda_1(\mathcal{L}_{V_i})+s_V\varepsilon+h\varepsilon)(t-t_{\varepsilon})}}{\lambda_1(\mathcal{L}_{V_i})+s_V\varepsilon+h\varepsilon}\right)\right)\right)\\
    \leq & I(x,t)\\
    \leq & H(x)\left(1-\exp\left(-\beta_{VH}\left(\displaystyle\int_0^{t_{\varepsilon}}V_i(x,s)\mbox{d}s+\dfrac{\underset{x\in\Omega}{\max}V(x,t_{\varepsilon})\underset{\Omega}{\max}\varphi_{s,1}}{\lambda_1(\mathcal{L}_{V_s})-h\varepsilon}\right)\right)\right),
 \end{array}
\]
hence, letting $t\to+\infty$ and integrating over $\Omega$,
\[
 \begin{array}{ll}
    & \displaystyle\int_{\Omega}H\exp\left(-\beta_{VH}\left(\displaystyle\int_0^{t_{\varepsilon}}V_i(\cdot,s)\mbox{d}s+\dfrac{\underset{x\in\Omega}{\max}V(x,t_{\varepsilon})\underset{\Omega}{\max}\varphi_{s,1}}{\lambda_1(\mathcal{L}_{V_s})-h\varepsilon}\right)\right)\\
    \leq & \displaystyle\int_{\Omega}H-I_\infty\\
    \leq & \displaystyle\int_{\Omega}H\exp\left(-\beta_{VH}\left(\displaystyle\int_0^{t_{\varepsilon}}V_i(\cdot,s)\mbox{d}s+\dfrac{\underset{x\in\Omega}{\min}V_{i}(x,t_{\varepsilon})\underset{\Omega}{\min}\varphi_{i,1}}{\lambda_1(\mathcal{L}_{V_i})+s_V\varepsilon+h\varepsilon}\right)\right).
 \end{array}
\]

If $|P_0-r_P/s_P|\leq\varepsilon$ and $V_0\leq\varepsilon$, then $t_{\varepsilon}=0$ and the estimates follow.

If $\lambda_1(\mathcal{L}_{V_s})<0$, by Theorem \ref{largetimebehaviour}, $I_\infty=H$, hence $\displaystyle\int_{\Omega}H-I_\infty=0$.

\section{Numerical study of the optimal control problem}\label{spatialhomogenization}

Recall that the optimal control problem is enunciated as follows: assuming that $r_V=r_V^{\text{field}}+r_V^{\text{refuge}}R$, $r_P=r_P^{\text{field}}+r_P^{\text{refuge}}R$, $H=H^{\text{field}}(1-R)$ with $R:\Omega\to[0,1]$ a spatial distribution of biodiversity refuges and with positive constants $r_V^{\text{field}}$, $r_V^{\text{refuge}}$, $r_P^{\text{field}}$, $r_P^{\text{refuge}}$, $H^{\text{field}}$, characterize the set of optimal $R$ such that the harvest of healthy beets $\lim_{t\to+\infty}\int_{\Omega} H(x)-I(x,t)\textup{d}x$ is maximized.

Although we worked in the previous analytical sections in a smooth setting (smooth domain, smooth functions $r_V$, $r_P$ and $H$, classical solutions) for the sake of simplicity, we actually have in mind, application-wise, a slightly less regular setting, where $\Omega\subset\mathbb{R}^2$ is a rectangle $[0,l_1]\times[0,l_2]$, where $R$ is an indicator function and where solutions are weak solutions. The extension of the preceding analytical sections to this setting by regularization procedures is a classical but technical issue and is deliberately not written in this paper. Our numerical simulations below will illustrate, quite unambiguously, that the analytical study of the optimal control problem with purely deterministic tools is likely vain, independently of the spatial dimension and the smoothness of the setting: indeed, the main issue lies in the strong dependency on the initial conditions. Therefore, in this section, we only present numerical simulations of the biologically more realistic setting.

\subsection{First insights}

Theorem \ref{largetimebehaviour} established the importance of the principal eigenvalue $\lambda_1(\mathcal{L}_{V_s})$. For the agro-ecological control strategy to succeed, this eigenvalue must be larger than $0$. Moreover, when this sign condition is satisfied, the rate of exponential decay of the vector population is, roughly speaking, $-\lambda_1(\mathcal{L}_{V_s})$: the larger the eigenvalue, the faster the eradication of vectors.
Therefore it is natural to consider the maximization of this eigenvalue as the spatial distribution of refuges $R$ varies. Since $R$ is in all cases an indicator function $\mathbf{1}_A$, with $A$ the refuge subdomain, the problem reduces to maximizing the eigenvalue as $A$ varies. Also, thanks to the ideal free distribution of predators which makes their equilibrium equal to $r_P/s_P$, the principal eigenvalue $\lambda_1(\mathcal{L}_{V_s})=\lambda_1(A)$ depends on $A$ in a completely explicit way:
\begin{align*}
\lambda_1(A) & =\lambda_1\left(-\sigma_V\Delta-\left(r_V^{\text{field}}-\frac{h}{s_P}r_V^{\text{field}}\right)-\left(r_V^{\text{refuge}}-\frac{h}{s_P}r_V^{\text{refuge}}\right)\mathbf{1}_A\right) \\
& = -\left(r_V^{\text{field}}-\frac{h}{s_P}r_V^{\text{field}}\right)+\lambda_1\left(-\sigma_V\Delta-\left(r_V^{\text{refuge}}-\frac{h}{s_P}r_V^{\text{refuge}}\right)\mathbf{1}_A\right) .
\end{align*}

The first remark to be made on this spectral optimization concerns the effect of fragmentation. Is a unique big connected refuge better than a multitude of small disconnected uniformly distributed refuges? The answer is negative: the eigenvalue $\lambda_1(A)$ is increasing with the frequency of the refuge $A$ (in dimension $1$, the frequency is the number of connected components, we will provide a preciser definition of the frequency later). This can be easily proved by periodically extending $A$ in $\mathbb{R}^2$ (recall that $\Omega$ is now a rectangle), by a change of spatial variable and by using the variational formula for $\lambda_1(A)$: calling $\mu=-r_V^{\text{refuge}}+\frac{h}{s_P}r_V^{\text{refuge}}$, and up to an additive constant, the principal eigenvalue corresponding to the refuges $A_n$ of frequency $n$ we denote $\lambda_1(A_n)$ is
\[
 \lambda_1(A_n)=\inf_{\varphi\in H^1(\Omega),\Vert\varphi\Vert_{L^2}=1}\displaystyle\int_{\Omega}n^2\sigma_V|\nabla\varphi|^2+\mu\mathbf{1}_{A_1}\varphi^2.
\]

The monotonicity of $\lambda_1(A_n)$ with respect to the frequency $n$ is hereby clear.

Moreover, the limit as the frequency becomes infinite while the area is fixed is identified thanks to a well-known homogenization result (Theorem 2.1 of \cite{chenlou}): indeed it is
\[
-r_V^{\text{field}}+\frac{hr_P^{\text{field}}}{s_P}+\frac{|A|}{|\Omega|}\left(-r_V^{\text{refuge}}+\frac{hr_P^{\text{refuge}}}{s_P}\right)=\left(1-\frac{|A|}{|\Omega|}\right)\lambda_1(\emptyset)+\frac{|A|}{|\Omega|}\lambda_1(\Omega).
\]

This limit tells, on a linear level, that the system with a refuge of high frequency behaves like it would if flowers and beets had been perfectly mixed together, and that this is the best configuration of flowers one can hope for.

The second remark follows naturally from this explicit limit: in the high frequency asymptotic regime, the exact shape of the connected components of $A$ does not matter and increasing the area of the refuge is beneficial.

Nevertheless, the quantity we actually want to maximize is not $\lambda_1$ but rather the harvest $\Phi=\lim_{t\to+\infty}\int_{\Omega} H(x)-I(x,t)\textup{d}x$. Even though Theorem  \ref{largetimebehaviour} hints toward a monotonic relationship between $\lambda_1$ and $\Phi$, it turns out this is false. The following argument makes it clear: the best possible eigenvalue corresponds to $A=\Omega$; however, in such a case, there is no beets to harvest anymore whence $\Phi=0$. Therefore the optimal refuge area is intermediate: $0<|A|<|\Omega|$. An even more subtle fact is the following: our numerical simulations show that $\Phi$ is hugely impacted by the initial conditions $P_0$ and $V_0$. For instance, when $P_0=r_P/s_P$ and $V_0$ is supported inside a connected component of the refuge, say $A_{\text{init}}$, it is actually better to have $A_{\text{init}}$ as large as possible, which means that $A_{\text{init}}$ should be the only connected component of $A$, or in other words $A$ should have a low frequency.

In our understanding of the sugar beet agro-ecosystem, the initial condition for the aphids cannot be predicted and is mostly random. A natural perspective is then the probabilistic study of the expected harvest. However this is way outside the scope of this paper, and we leave it for future work. In particular, a more precise knowledge of the law of this random initial condition is necessary. Preliminary numerical simulations assuming a uniform law tend to indicate that the expected harvest is indeed increasing with the frequency of the refuge.

\subsection{Numerical framework}

We simulate the Cauchy--Neumann system \eqref{eqspring} with a classical semi-explicit scheme with finite differences (details can be found in the supplementary materials).

The script is written in Python, using the numpy \url{https://numpy.org/}, scipy \url{https://scipy.org/} and matplotlib \url{https://pypi.org/project/matplotlib/} packages.
For longer computations, this work was performed using the computing facilities of the CC LBBE/PRABI.
Details on the computation of the parameters can be found in the supplementary materials.

The refuge has area $\mathcal{A}\in(0,|\Omega|)$ and we consider a sequence $(A_n)_{n\in\mathbb{N}}$ of uniformly distributed refuges: 
\[
A_n:=\left(\bigcup\limits_{m=0}^{n-1}\left[\dfrac{m}{n}L,\dfrac{m}{n}L+\dfrac{\sqrt{\mathcal{A}}}{n}\right]\right)^2.
\]

For our simulation, we take $N=4000$ time steps, $J=80$ spatial steps, $T=365/4$ days, $\Omega=[0,L]^2$ with $L=300$ meters, $\mathcal{A}=\dfrac{|\Omega|}{5^2}$, $r_P=(r_R-r_F)\mathbf{1}_{A_4}+r_F$ and the initial conditions $I_0=0$, $V_{i,0}=(1/100)(r_V/s_V)$ distributed in 3 patches of different sizes, $V_{s,0}=(9/100)(r_V/s_V)$ distributed in the same patches than $V_{i,0}$, $P_0=r_P/s_P$.

\subsection{A remark on the ideal free dispersal strategy}

The growth rate $r_P$ is now a piecewise-constant function. This could appear as a difficulty, as the ideal free dispersal strategy $\nabla\cdot(r_P\nabla(P/r_P))$ is \textit{a priori} defined in \cite{cantrellcosnerlou2010} for $r_P$ at least of class $\mathcal{C}^2$. For piecewise-constant $r_P$, an alternative ideal free dispersal strategy with interface conditions is suggested in \cite{macielcosnercantrelllutscher}. However, when the space is discretized, even regular functions become piecewise-constant. In fact, it can be verified that an ideal free dispersal strategy of the form $\nabla\cdot(r_P\nabla(P/r_P))$ corresponds, after spatial discretization, to the interface conditions in \cite{macielcosnercantrelllutscher}.
Note also that the potential overflow issues can be overcome with a change of variable $\tilde{P}=P/r_P$. 

\subsection{Snapshots of the evolution}
Using the package animation \url{https://pypi.org/project/animation/}, we produce a video of the simulation (figures \ref{screensofsimu}, \ref{screensofsimu2}).

\begin{figure}
    \centering
    \includegraphics[scale=0.3]{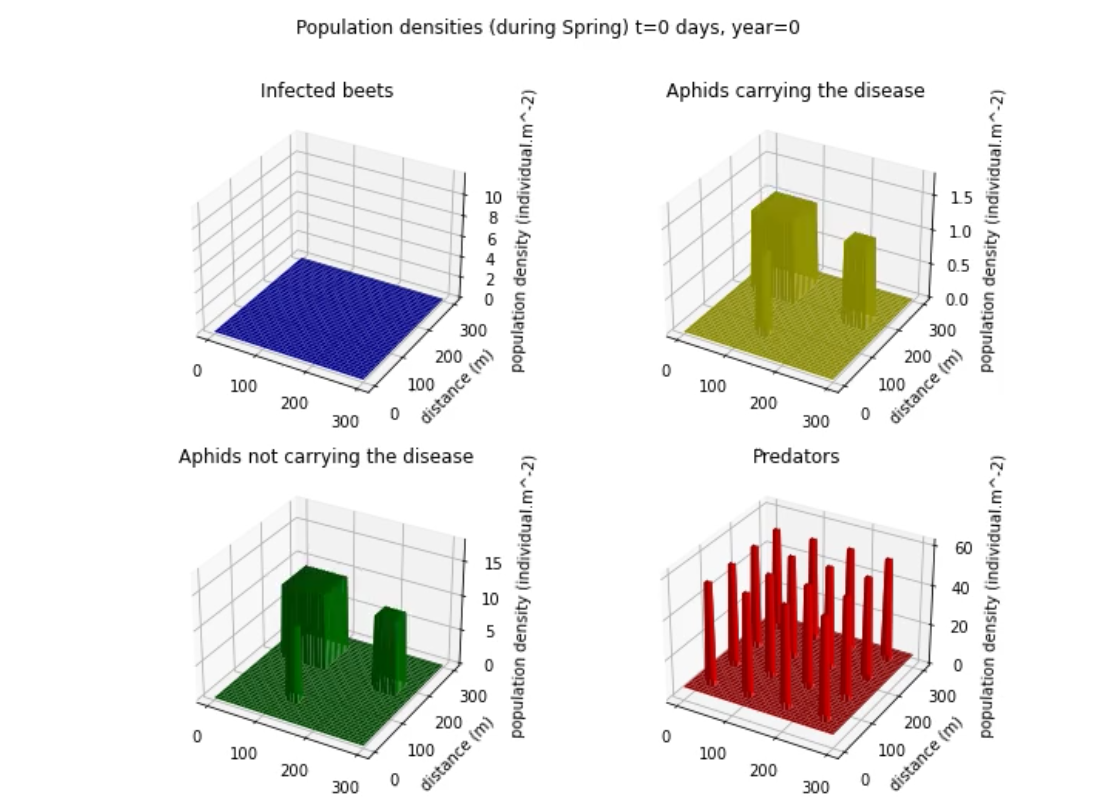}
    \caption{Video of the numerical simulation at the initial condition}
    \label{screensofsimu}
    \includegraphics[scale=0.3]{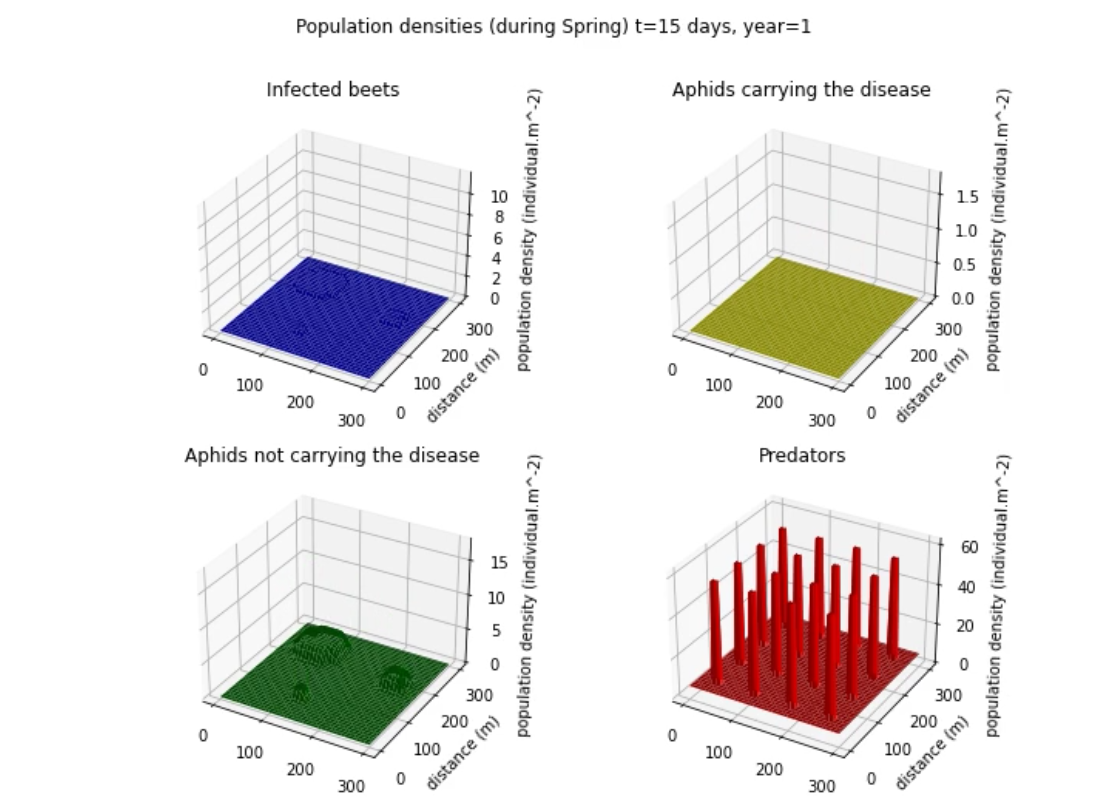}
    \caption{Video of the numerical simulation at 15 days}
    \label{screensofsimu2}
\end{figure}

At the end of the simulation, the total number of healthy beets is $1013554$ out of $1016470$. In the video, we can see that a very small portion of the beets are infected by aphids in their starting patches, while the aphids are quickly eradicated by the predators.

\subsection{The optimal control problem}

In the numerical simulations, we will refer to certain positionning of the populations as defined in the following.

\begin{definition}
 \begin{itemize}
  \item ``Frequency of the refuges'': the integer $n$, where $R=\mathbf{1}_{A_n}$, i.e $n^2$ squares uniformly distributed in the field, whose areas sum up to $\mathcal{A}$.
  \item ``Random patches'': the 3 patches one can see in the snapshot of the video at the initial condition.
  \item ``Centered patch'' = the frequency 1 square.
 \end{itemize}
\end{definition}

As we explained before, there is no monotonicity of the harvest with respect to $\lambda_1(\mathcal{L}_{V_s})$. To illustrate, we use the fact that increasing the frequency of the refuges increases $\lambda_1(\mathcal{L}_{V_s})$, as shown in \cite{chenlou}. Figure \ref{harvestfuncfreqpatchs} displays the harvest $\int_{\Omega}H-I_{\infty}$ as a function of the frequency of the refuges, with the same set of parameters and initial conditions as before. Here, the frequencies $8$ and $16$ are clearly worse than the lower frequencies. The phenomenon can be even clearer: to obtain Figure \ref{harvestfuncfreqcenter}, we use a single centered patch as the initial condition for $V$, at the same location as the refuges at frequency $1$. Obviously, in this case, the best strategy is to use the refuges of frequency $1$, in order to place the predators exactly where the aphids are at the beginning of the simulation. However, when $V_0$ is homogeneous in space, as seen in Figure \ref{harvestfuncfrequniform}, the harvest seems to be increasing with $\lambda_1(\mathcal{L}_{V_s})$, which could be an interesting perspective of this paper.

\begin{figure}
    \centering
    \includegraphics[scale=0.6]{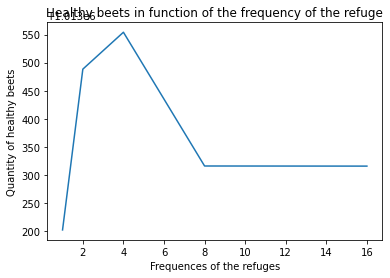}
    \caption{Healthy beets in function of the frequency of the refuges, with the initial condition of the aphids $V_0$ in random patches} 
    \label{harvestfuncfreqpatchs}
    \includegraphics[scale=0.6]{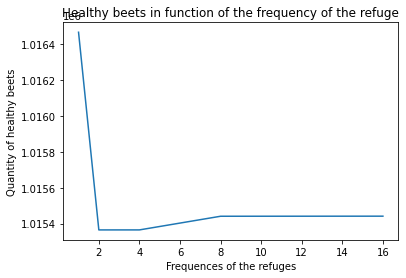}
    \caption{Healthy beets in function of the frequency of the refuges, with the initial condition of the aphids $V_0$ in a centered patch}
    \label{harvestfuncfreqcenter}
    \includegraphics[scale=0.6]{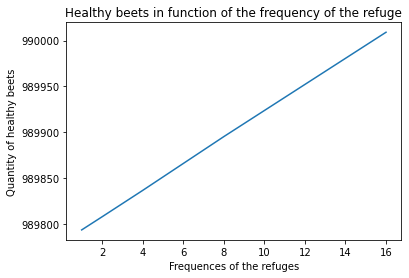}
    \caption{Healthy beets in function of the frequency of the refuges, with the initial condition of the aphids $V_0$ homogeneous in space}
    \label{harvestfuncfrequniform}
\end{figure}

To illustrate the important effect of the initial conditions, we numerically study the optimal quantity of refuges. Say $R$ is a real parameter in $[0,1]$. $r_V$, $r_P$ and $H$ are now homogeneous in space, and the problem is to find $R_{\text{opt}}\in[0,1)$ to maximize the quantity $\int_{\Omega}H-I_{\infty}$. In Figure \ref{harvestfuncquantlowVi} and \ref{harvestfuncquanthighVi}, we took the same set of parameters as in the supplementary materials, with the same initial conditions as in the numerical framework except that we took a uniform initial condition on $V$. In Figure \ref{harvestfuncquantlowVi}, $V_{i,0}$ is quite small, and the simulation hints that it is best not to put any resources in the field: $R_{\text{opt}}=0$. We can interpret this by saying that the predators handle the small quantity of aphids well enough so that they do not need more resources to eradicate the aphid population in a very short amount of time. On the contrary, in Figure \ref{harvestfuncquanthighVi}, $V_{i,0}$ is large, and the predators actually need resources for some beets to be saved: $R_{\text{opt}}$ is at a great distance from $0$. Depending on the initial conditions, it seems that $R_{\text{opt}}$ can be located at any point of $[0,1)$.

\begin{figure}
    \centering
    \includegraphics[scale=0.8]{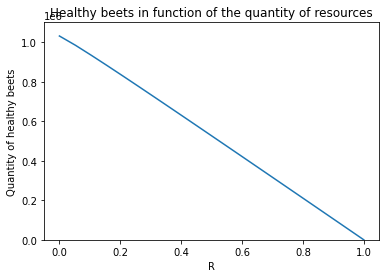}
    \caption{Healthy beets in function of the quantity of resources, with the initial condition of the infected aphids $V_{i,0}=\dfrac{r_V}{100s_V}$ and the refuges $R$ both homogeneous in space}
    \label{harvestfuncquantlowVi}
    \includegraphics[scale=0.8]{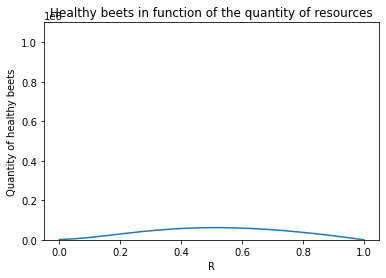}
    \caption{Healthy beets in function of the quantity of resources, with the initial condition of the infected aphids $V_{i,0}=\dfrac{2r_V}{s_V}$ and the refuges $R$ both homogeneous in space}
    \label{harvestfuncquanthighVi}
\end{figure}

\bibliographystyle{plain}
\bibliography{biblioBaptiste}

\end{document}